\documentclass[11pt]{article}
\usepackage[latin1]{inputenc}
\usepackage[english]{babel}
\usepackage{amssymb}
\usepackage{amsmath,amsthm}%usepackage para numera\c{c}\~{a}o das f\'{o}rmulas
\usepackage[normalem]{ulem}
\usepackage{cite}

\usepackage{physics}

\newtheorem{thm}{Theorem}[section]

\newtheorem{cor}{Corollary}[section]
\newtheorem{lem}{Lemma}[section]
\newtheorem{rmk}{Remark}[section]
\newtheorem{exmp}{Example}[section]

\makeatletter

\let \var=\varphi
\let \vare=\varepsilon

\let \de=\delta
\let \th=\theta
\let \la=\lambda

\let \p=\partial
\let \q=\quad
\let \qq=\qquad

\let \smal=\smallskip
\let \dps=\displaystyle
\let \ul=\underline
\let \ul=\underline
\let \ol=\overline

\newcommand{\R}{\mathbb{R}}
\newcommand{\N}{\mathbb{N}}

\begin{document}

 %Draft - February 28, 2024 

\begin{center}
%\textbf{\Large{Permanence and stability  for generalized nonautonomous Nicholson systems - ??}}
	\textbf{\Large{Periodic solutions for a  delayed competitive chemostat model with  periodic  nutrient input and rate }}
	\end{center}

 \centerline{\scshape Teresa Faria\footnote{Work   supported by  FCT-Funda\c c\~ao para a Ci\^encia e a Tecnologia (Portugal) under project UIDB/04561/2020.}}
 
 \smal
 \centerline{
 Departamento de Matem\'atica and CMAFcIO, Faculdade de Ci\^encias, Universidade de Lisboa}
  \centerline{
Campo Grande, 1749-016 Lisboa, Portugal}
 
\centerline{  Email:
teresa.faria@fc.ul.pt}

\vskip .5cm

\begin{abstract}  
A   nonautonomous  periodic chemostat model with delays modelling $n$ species in competition  is considered. Sufficient conditions on the coefficients and consumption functions for the species are given, for both the extinction of  the species and for the existence of $n$ nontrivial and nonnegative periodic solutions.  Further criteria guarantee that the system admits at least one strictly positive periodic solution.
%A key ingredient for the analysis is the derivation of a conservation principle.
\end{abstract}



\section{Introduction}
\setcounter{equation}{0}

A chemostat is a  bioreactor used in laboratories  for the culture of microorganisms, in which the chemical environment is kept steady, simulating what occurs in nature.  It is composed of a  vessel with a culture liquid and two pumps.
 The   culture liquid  contains a mixture of nutrients and growing  microorganisms, whereas the pumps keep its
 volume   constant by a continuous feed  of fresh substrate  and effluent of the culture. %, at a fixed rate.
 % ,  at a fixed rate, and  contains nutrients and growing  microorganisms.
 % The vessel of the chemostat contains liquid with substrate and one or more microorganisms of consumers, and continuous inputs  of  nutrients  and outputs of the culture,  at a fixed rate,  keeps the volume  constant. 
 Chemostats are used in research (biology, ecology, mathematics) and industry. 
  For the last fifty years,  an array of different  models for the growth of species in  chemostats  has been proposed  and extensively analysed by mathematicians and biologists.    
The literature on classic ordinary differential equation (ODE)  chemostat systems with several species in competition is quite extensive, and includes both autonomous and nonautonomous models -- see the monographs by Smith and Waltman \cite{SmithW} and Zhao \cite{Zhaobook}, where detailed biological explanations can be found, and the selected works \cite{Fan,Smith81,SmithTh,WZ}. 
%More recently,
 %researchers have introduced delays
 Models become more realistic with the introduction of delays, which  reproduce  the maturation time for the cells of the consumer species, reflecting the time lag between consumption of nutrients and growth. Delayed autonomous chemostat models were  first proposed by Freedman et al. \cite{FSW,FSW2} for a single species and by Ellermeyer \cite{Eller} for two species. For more references on chemostat models with delays, see e.g.  \cite{Eller2,JiaZ,MRM,WX} and references therein.
 %. The literature is very extensive, here we only refer to some selected works, 
 %See   \cite{ARS20,ARS20b,ARS23,Eller,Eller2,Faria23,MRM} and references therein.  
 %More biological explanations can be found in \cite{...,SmithW} and references therein.

%  The literature on classic ODE  chemostat systems with several species in competition is quite extensive, and includes both autonomous and nonautonomous models. The models become more realistic with the introduction of delays, which reflect the time lag between consumption of nutrients and  the maturation time for the cells of the consumer species.  Several autonomous models have been proposed, see Ellermeyer... 

% A chemostat is a bioreactor composed by a vessel and two pumps, in which the chemical environment is steady, simulating what occurs in nature. 
% The  volume of  the culture liquid in the vessel is kept constant by continuous inputs  of fresh nutrients  and outputs of the culture,  at a fixed rate, and  contains nutrients and growing  microorganisms. Chemostats are used in research (biology, ecology, mathematics) and industry. 
In \cite{Eller}, Ellermeyer derived and studied the delayed  chemostat model with two microorganisms in competition given by
\begin{equation}\label{Eller}
 \begin{split}
 S'(t)&=d(S_0-S(t))-p_1(S(t))x_1(t)-p_2(S(t))x_2(t)\\
x_i'(t)&=-dx_i(t)+e^{-d\tau_i}p_i(S(t-\tau_i))x_i(t-\tau_i),\q i=1,2,\end{split}
\end{equation} 
with $S_0, d,\tau_1,\tau_2$ positive constants and $p_1,p_2\ge 0$ smooth increasing functions.
In \eqref{Eller},   $S_0$ represents the initial concentration of the substrate in the chemostat liquid free of consumers and $d$ the rate of the input of fresh nutrients; at a time $t$,
  $x_i(t)$ is the biomass for the consumer species $i$ in the vessel,  $\tau_i$ is the time for its cell-cycle  and  $p_i(x)$ its ``uptake" or ``response" function,  $i=1,\dots,n$.
Quite often,  Michaelis-Menten functional responses $p_i(x)=\frac{b_ix}{1+k_ix}$ are used in the models, but most results are valid for other strictly increasing, bounded functions. See  \cite{SmithW} for more biological insights. System \eqref{Eller} was later generalised in \cite{WX}, for the case of $n$  microorganisms. 

Adjusting the derivation of Ellermeyer's classic model  \cite{Eller} to the situation of  time-dependent  washout rate $d(t)$ and  input concentration of  substrate  $S_0(t)$,    one  derives in a straightforward way  the system below, where now the delays $ \tau_i$  appear as {\it distributed delays} in the coefficients:
\begin{equation}\label{Chem}
 \begin{split}
 S'(t)&=d(t)(S_0(t)-S(t))-\sum_{i=1}^np_i(S(t))x_i(t)\\
x_i'(t)&=-d(t)x_i(t)+e^{-\int_{t-\tau_i}^td(s)\, ds}p_i(S(t-\tau_i))x_i(t-\tau_i),\q i=1,\dots,n,\\
\end{split}
\end{equation}
%In this paper, we study the generalised  version of the   Ellermeyer's model with   time-varying washout rate $d(t)$ and  input concentration $S_0(t)$, given by 
where $t\ge t_0$, $S_0,d, p_i$ are continuous functions, $S_0(t),d(t)$ are  positive   and  $p_i(x)$ is positive for $x>0,$  $p_i(0)=0,$ $ i=1,\dots,n$. Conf. \cite{ARS20} (for  $n=1$) and \cite{Faria21} (for  $n\ge 1$). Without loss of generality, we set the initial time $t_0=0$. %The study of \eqref{Chem} was initiated in \cite{Faria23}.
The analysis of nonautonomous delayed competitive chemostat systems of the form \eqref{Chem}  was initiated by the author in \cite{Faria23}, where the partial or total extinction of the species growing in the chemostat versus their permanence were studied.
The aim of this paper is to study  the periodic case of system \eqref{Chem}, i.e., when
$S_0(t),d(t)$ are  periodic with a common period.
The  particular model with $S_0(t)\equiv S_0>0$,
 \begin{equation}\label{Chem0}
 \begin{split}
 S'(t)&=d(t)(S_0-S(t))-\sum_{i=1}^np_i(S(t))x_i(t)\\
x_i'(t)&=-d(t)x_i(t)+e^{-\int_{t-\tau_i}^td(s)\, ds}p_i(S(t-\tau_i))x_i(t-\tau_i),\q i=1,\dots,n,\\
\end{split}
\end{equation} 
will receive further attention, as it has special properties. For instance, 
one easily sees that the ordered interval ${\cal S}:=[0,S_0]\times (\R^+)^n\subset C^+$ is positively invariant for \eqref{Chem0} and that $(S_0,0,\dots,0)$ is always an equilibrium of \eqref{Chem0}, often called the {\it washout equilibrium}. 
%In \eqref{Chem},  at a time $t$, $S_0(t)$ represents the initial concentration of the substrate in the chemostat liquid free of consumers and $d(t)$ the rate of the input of fresh nutrients;
%  $x_i(t)$ is the biomass for the consumer species $i$ in the vessel,  $\tau_i$ is the time for its cell-cycle  and  $p_i(x)$ its ``uptake" or ``response" function,  $i=1,\dots,n$.
%Quite often,  Michaelis-Menten functional responses $p_i(x)=\frac{b_ix}{1+k_ix}$ are used in the models, but most results are valid for other strictly increasing, bounded functions. See e.g. \cite{SmithW} for more biological explanations. Adjusting the derivation of Ellermeyer's model  \cite{Eller} to the situation of time-dependent  washout rate and  input concentration of  substrate,    one  derives in a straightforward way that now the delays $ \tau_i$  appear as {\it distributed delays} in the coefficients of system  \eqref{Chem}. See e.g. \cite{ARS20} for the case $n=1$.

Recently, several  nonautonomous chemostat models with a single growing species and a discrete delay  have been proposed and analysed, as in   \cite{ARS20,ARS20b}. 
As far as the author knows, 
%very little has been done regarding  
%To our knowledge, 
however,  with the exception of \cite{JiaZ} there are no previous results in the literature concerning periodic solutions to  periodic chemostat systems with delays and several growing species in competition.

%The study of \eqref{Chem} was initiated in
%The analysis of nonautonomous delayed competitive chemostat systems of the form \eqref{Chem}  was initiated by the author in \cite{Faria23}, where the partial or total extinction of the species growing in the chemostat versus their permanence were studied, without assuming that the coefficients were periodic.

Motivated by the work in   \cite{Faria23},
 here   $S_0(t)$ and $d(t)$ are assumed to be $\omega$-periodic functions (for some $\omega>0$). A first aim of this paper is to refine the criteria in \cite{Faria23} for
 the asymptotic behaviour of solutions of \eqref{Chem}, in terms of either  extinction of all species or  competitive exclusion. The main goal, however, is to establish  the existence of nontri\-vial and nonnegative (and if  possible strictly positive) $\omega$-periodic solutions. These $\omega$-periodic solutions will be found as fixed points of a convenient operator acting on a closed cone of a Banach space, by using the Krasnoselskii  fixed point theorem. We recall that  in the case of permanence of positive solutions, arguments based  on Schauder fixed point theorem have been used to find  positive $\omega$-periodic solutions of delay differential equations (DDEs) \cite{Faria17, Zhao}.  Following this approach, by using a Poincar\'e map here  we shall see that, if the (uniform) persistence  of positive solutions is  imposed as an additional hypothesis, then  our main criterion on the existence of nonnegative and nontrivial $\omega$-periodic solutions also implies that at least one strictly positive $\omega$-periodic solution must exist.  
A keystone for the work presented here is  the  {\it conservation principle} stated in \cite{Faria23} (see  \cite{SmithW} for the ODE model), which enables us to eliminate the variable $S(t)$ and replace system \eqref{Chem} by an $n$-dimensional functional differential equation (FDE) of advanced-delayed type. 
Besides the conservation principle and  fixed point techniques, some major arguments  in this paper are based on comparison results from the theory of quasi-monotone FDEs \cite{Smith}.

Some alternative nonautonomous delayed chemostat models for  $n$ competing species  have been proposed. Mazenc et al. \cite{MRM} derived a  model with constant input rate and initial concentration of the nutrient, for which the stability of an asymptotic positive equilibrium and its robustness  were studied. 
On the other hand, the continuation theorem of the Leray-Schauder degree theory has  been used to find positive  periodic solution for  delayed chemostat models, either for the case of a single  \cite{ARS20,ARS20b} or for two microorganisms \cite{JiaZ} in the chemostat. In fact, in \cite{ARS20}, Amster et at.  considered the model \eqref{Chem} with $n=1$, whereas
in
\cite{ARS20b} the same authors studied the model with constant input rate $d>0$ and delay $\tau>0$ given by
\begin{equation}\label{ARS20b}
 \begin{split}
 S'(t)&=d(S_0(t)-S(t))-p(S(t))x(t)\\
x'(t)&=-dx(t)+p(S(t-\tau))x(t),\end{split}
\end{equation}
where $S_0(t)$ is continuous and periodic. In both papers,  additional requirement on $p(x)$ are imposed.
%$p(x)$ is supposed to be $C^1$-smooth with $p(0)=0,p'(x)>0$.
 In \cite{ARS20b}, the reader can find
necessary and sufficient conditions for the existence of a positive periodic solution for \eqref{ARS20b}, which is unique for $\tau$ small. In \cite{JiaZ}, a nonautonomous chemostat system with two species involving distributed infinite delays and periodic nutrient  was studied.%
%
%\begin{equation}\label{ARS20b2}
% \begin{split}
% S'(t)&=d(S_0(t)-S(t))-p(S(t))x(t)\\
%x'(t)&=-dx(t)+e^{-d\tau}p(S(t-\tau))x(t-\tau),\end{split}
%\end{equation}
%

A summary of the topics addressed in this paper is now given.  In Section 2, after recalling  the conservative principle and other preliminary results, a refinement of the criterion in \cite{Faria23} for the globally asymptotically stable of the so-called {\it washout solution} is given --  in biological terms, this means the extinction of all species.
Section 3 contains the main results of the paper on the existence of  periodic nonnegative and nontrivial solutions.  %of the form $(S_i^*(t),0,\dots, x_i^*(t),\dots, 0)$ with $S_i^*(t), x_i^*(t)$ strictly positive ($1\le i\le n$). 
Section 4 deals with two major subjects: (i) competitive exclusion, when only  the  fittest species survives in a sustained persistent way and the others are driven to extinction; (ii)
periodic coexistence, in the sense of  criteria for the existence of at least a strictly positive $\omega$-periodic solution of \eqref{Chem}. %Namely, the main result Theorem \ref{thm4.4} is established. 
Some examples illustrate the results. The last section includes some conclusions, open problems and directions for future work. 

\section{Preliminary results and extinction}
\setcounter{equation}{0}

Some notation is first introduced.  For $\tau=\max_{1\le i\le n} \tau_i$
%For  a function $f:\R^+\to \R$ bounded, we denote
%$$\ol f=\sup_{t\ge 0}f(x),\q \ul f=\inf_{t\ge 0}f(x).$$ Let $\tau=\max_{1\le i\le n} \tau_i$. 
and  $N\in\N$, consider the Banach space
$C([-\tau,0],\R^{N})$ equipped with the supremum norm $\|\var\|=\max_{\th\in [-\tau,0]}|\var(\th)|$, for a fixed norm  in $\R^{N}$.
  $\R^{N}$ is also  interpreted as the  subset of $C([-\tau,0],\R^N)$ of constant functions. A vector $v=(v_1,\dots,v_N)\in\R^N$ is positive if $v_i>0$ for all $i=1,\dots,N$; and $v\ge 0$  if $v_i\ge 0$ for all $i=1,\dots,N$. For a DDE in $C([-\tau,0],\R^N)$ of the form $x'(t)=f(t,x_t)$, $dom (f)=[0,\infty)\times  C([-\tau,0],\R^N)$,  and a set ${\cal S}\subset C([-\tau,0],\R^N)$, an  initial condition $\var$ in $\cal S$ at time $\sigma$ is written as $x_\sigma=\var$, where, as usual, $x_t \, (t\ge 0)$ denotes  the element in $C([-\tau,0],\R^N)$ defined by
$$x_t(\th)=x(t+\th),\q \th\in [-\tau,0].$$ Naturally, the phase space for  \eqref{Chem} is 
$C:=C([-\tau,0],\R^{n+1})$. Consider also the nonnegative cone $C^+=\{\var\in C: \var\ge 0\}$, where $\ge$ is the usual partial order in $C$: for $\var,\psi\in C$, $\var\ge \psi$ if $\var(\th)\ge \psi(\th)$ for all $\th\in [-\tau,0]$.  
Due to the biological interpretation of \eqref{Chem},   either $C^+$ or $C_0^+=\{ \var\in C^+:\var(0)>0\}$ will be taken as the set of initial conditions.  
%One easily sees \cite{Faria23} that solutions of \eqref{Chem} with initial conditions in  $C^+$ (respectively $C_0^+$) are defined and  nonnegative (respectively positive) on $[0,\infty)$.

For \eqref{Chem}, the following hypotheses will be assumed:
\begin{itemize}
	\item[(H1)]  $d,S_0:\R^+\to\R$ are continuous, positive  and $\omega$-periodic functions (for some $\omega>0$);
\item[(H2)] $p_i:\R^+\to\R$ are  non-decreasing, locally Lipschitzian with $p_i(0)=0, p_i(x)>0$ for $x>0,  i=1,\dots,n$.
\end{itemize}

If (H1)-(H2) hold, from the results in 
 \cite{Faria23} it follows  that the solutions with initial conditions $x_0=\var\in C^+$ are defined on $[0,\infty)$, and the sets $C^+$ and $C_0^+$ are positively invariant for \eqref{Chem}.

%Consider again the chemostat model \eqref{Chem} in $C=C([-\tau,0],\R^{n+1})$, where now 
%% \begin{itemize} \item[(H3)] 
%the functions $ d(t), S_0(t)$ are $\omega$-periodic ($\omega>0$) positive continuous functions. 
%%\end{itemize}

Denote $\ol f=\sup_{t\ge 0}f(t), \ul f=\inf_{t\ge 0}f(t)$ for $f$ bounded on  $[0,\infty)$, so that $0<\ul{d}\le d(t)\le \ol{d}<\infty, \ 0<\ul{S_0}\le S_0(t)\le \ol{S_0}<\infty$. 
 For   \eqref{Chem},  and using the terminology   in \cite{SmithW} for the case of
a  chemostat model without delays, the following   {\it ``conservation principle"}  established in \cite{Faria23} applies to the periodic model above:

\begin{lem}\label{lem2.1}  Assume (H1)-(H2). For  a nonnegative solution $(S(t),x_1(t),\dots,x_n(t))$ of \eqref{Chem},  define
 \begin{equation}\label{y}
 y(t)=S(t)+\sum_{i=1}^n e^{\int_t^{t+\tau_i} d(s)\, ds}x_i(t+\tau_i).
\end{equation}
Then $y(t)$ is a positive solution of the ODE
 \begin{equation}\label{y_ode}
y'=d(t)\big (S_0(t)-y\big),
\end{equation}
and satisfies
$$\ul{S_0}\le \liminf_{t\to \infty} y(t)\le  \limsup_{t\to \infty} y(t)\le \ol{S_0}.$$
 \end{lem}

 \begin{rmk}\label{rmk2.0}  {\rm A  {\it ``conservation principle"}   holds for more general nonautonomous chemostat models, %of the form
 \begin{equation}\label{Chem_b}
 \begin{split}
 S'(t)&=d(t)(S_0(t)-S(t))-\sum_{i=1}^nb_i(t)p_i(S(t))x_i(t)\\
x_i'(t)&=-d(t)x_i(t)+e^{-\int_{t-\tau_i}^td(s)\, \dd s}b_i(t-\tau_i)p_i(S(t-\tau_i))x_i(t-\tau_i),\q i=1,\dots,n,\\
\end{split}
\end{equation}
for which Lemma \ref{lem2.1} is still valid with $y(t)$ given by the same formula \eqref{y}.}
\end{rmk}

% This auxiliary result is very useful,  
% as the nonhomogeneuous linear scalar ODE \eqref{y_ode} allows   the reduction of \eqref{Chem} to  an $n$-dimensional system. 
%\begin{proof} For $y(t)$ in \eqref{y} , The derivative  along solutions of \eqref{Chem} leads to  \eqref{y_ode}.
%\end{proof}

From the conservation principle, it follows  that \eqref{Chem} is dissipative, with $ \limsup_{t\to \infty} S(t)\le \ol{S_0},  \limsup_{t\to \infty} x_i(t)\le e^{-\ul{d}\tau_i} \, \ol{S_0}$ for $i=1,\dots,n.$
In addition,  from \cite{Faria23}, there is $m_0>0$ such that the coordinate $S(t)$  of all nonnegative solutions satisfies $\liminf_{t\to\infty} S(t)\ge m_0$.
%\begin{equation}\label{S^-}S^-:=\liminf_{t\to\infty} S(t)\ge m_0.\end{equation}
For specific functions $p_i(x)$, an explicit estimative  for $m_0$ can be provided. Moreover, if condition 
 \begin{equation}\label{pers}p_i(m_0)>\max_{t\in [0,\omega]} \Big (d(t)e^{\int_{t-\tau_i}^td(s)\, ds}\Big),\q i=1,\dots,n,\end{equation} is satisfied, it was also established in \cite{Faria23} that \eqref{Chem} is uniformly persistent (in $C_0^+$).

Observe that there is a unique positive $\omega$-periodic  solution $y^*(t)$ of \eqref{y_ode}, given by
\begin{equation}\label{y_per}y^*(t)=y^*(0) e^{-\int_0^t d(s)\, ds}+\int_0^td(s)e^{-\int_s^t d(r)\, dr}S_0(s)\, ds\end{equation} and 
\begin{equation}\label{y_per2}y^*(0)=\big(e^{\int_0^\omega d(r)\, dr}-1\big)^{-1}\int_0^\omega d(s)e^{\int_0^sd(r)\, dr}S_0(s)\, ds.\end{equation}
Moreover, all the other solutions $y(t)$ satisfy $y(t)-y^*(t)\to 0$ as $t\to \infty$.
The periodic function $ (y^*(t),0,\dots,0)$ is always a solution of  \eqref{Chem}, which in this context is called the periodic  ``washout solution" or the ``trivial periodic  solution". For \eqref{Chem0}, the washout solution is the washout equilibrium $ (S_0,0,\dots,0)$.

%\begin{equation}\label{Chem_y}
%x_i'(t)=-d(t)x_i(t)+f_i(t,x_t),\q i=1,\dots,n, 
%\end{equation}
%for
%\begin{equation}\label{f_i} f_i(t,x_t)=e^{-\int_{t-\tau_i}^t d(r)\, \dd r}p_i\Big(y(t-\tau_i)-\sum_{j=1}^n e^{\int_{t-\tau_i}^{t+\tau_j-\tau_i} d(s)\, \dd s}x_j(t+\tau_j-\tau_i)\Big)x_i(t-\tau_i),\end{equation} 
%where now $x_t\in \mathfrak{C}=C([-\tau,\tau^0]; \R^n)$ with $ \tau=\max_{1\le i\le n}\tau_i$, $\tau^0=\max_{1\le i,j\le n} (\tau_i-\tau_j)$, $x_t(\th):=x(t+\th)$ for $\th\in[-\tau,\tau^0]$. 
%
%
%

As the first equation in \eqref{Chem} has the form $S'(t)=-d(t)S(t)+g_0(t,x_t)$ with $g_0(t,\var)$ not necessarily positive, when looking for periodic solutions of \eqref{Chem} the framework in \cite{FF} and several other papers is not directly applicable.  A way to tackle this problem is to use the conservation principle and  the nonhomogeneuous linear ODE  \eqref{y_ode} as an auxiliary equation, and thus reduce \eqref{Chem} to  an $n$-dimensional system.  This reduction is performed below.

If  $(S(t),x_1(t),\dots,x_n(t))$ is a periodic solution of \eqref{Chem}, then  for $y(t)$ in \eqref{y} we have $y(t)=y^*(t)$ given by \eqref{y_per}-\eqref{y_per2}. We therefore eliminate the first equation in \eqref{Chem} and replace $S(t)$ by\begin{equation}\label{S(t)} S(t)=y^*(t)-\sum_{i=1}^n e^{\int_t^{t+\tau_i} d(s)\, ds}x_i(t+\tau_i),\end{equation}  in the last $n$ equations. Instead of looking for nonnegative  periodic solutions of \eqref{Chem}, equivalently  the aim is to find nonnegative  periodic solutions of the following  functional differential system of {\it mixed-type} (delayed and advanced-type):
  \begin{equation}\label{Chem_y}
x_i'(t)=-d(t)x_i(t)+f_i(t,x_t),\q ,\ i=1,\dots,n,
\end{equation}
where now $x_t\in \mathfrak{C}=C([-\tau,\tau^0]; \R^n)$ with $ \tau=\max_{1\le i\le n}\tau_i$, $\tau^0=\max_{1\le i,j\le n} (\tau_i-\tau_j)$, $x_t(\th):=x(t+\th)$ for $\th\in[-\tau,\tau^0]$ and
  \begin{equation}\label{f_i}f_i(t,x_t)=e^{-\int_{t-\tau_i}^td(s)\, ds}p_i\bigg(y^*(t-\tau_i)-\sum_{j=1}^n e^{\int_{t-\tau_i}^{t+\tau_j-\tau_i} d(s)\, ds}x_j(t+\tau_j-\tau_i)\bigg)x_i(t-\tau_i).
\end{equation}
Existence, uniqueness, positiveness  and boundedness on $[0,\infty)$ of solutions for \eqref{Chem_y} is assured by the original setting for system \eqref{Chem}.
From \eqref{y}, we have $y^*(t-\tau_i)-\sum_{j=1}^n e^{\int_{t-\tau_i}^{t+\tau_j-\tau_i} d(s)\, ds}x_j(t+\tau_j-\tau_i)\ge 0$ for all $i$ and $t\ge 0$.
If  all the delays are equal, $\tau_i=\tau$ for all $i$, then \eqref{Chem_y} is in fact a delayed  differential system; otherwise, system \eqref{Chem_y}  is of  mixed-type, since it has both delayed and advanced arguments.  %With the obvious meaning, 
Henceforth, $\mathfrak{C}$ is endowed with the supremum norm, $\mathfrak{C}^+$ is the cone  of nonnegative elements of  $\mathfrak{C}$, and  $\le $ is the partial order in $\mathfrak{C}$ defined by $\phi\le \psi$ if $\psi-\phi\in \mathfrak{C}^+$.

%--
%
%System \eqref{Chem} is equivalent to the  system of  {\it mixed} (delayed-advanced) functional differential equations
%\begin{equation}\label{mixed}
%\begin{split}
%y'(t)&=d(t)\big (S_0(t)-y(t)\big)\\
%x_i'(t)&=-d(t)x_i(t)+\tilde f_i(t,x_t),\qq i=1,\dots,n,
%\end{split}
%\end{equation}
%where now $x_t\in \mathfrak{C}=C([-\tau,\tau^0]; \R^n)$ with $ \tau=\max_{1\le i\le n}\tau_i$, $\tau^0=\max_{1\le i,j\le n} (\tau_i-\tau_j)$, $x_t(\th):=x(t+\th)$ for $\th\in[-\tau,\tau^0]$, and
%\begin{equation*}\label{f_i} \tilde f_i(t,x_t)=e^{-\int_{t-\tau_i}^t d(r)\, dr}p_i\Big(y(t-\tau_i)-\sum_{j=1}^n e^{\int_{t-\tau_i}^{t+\tau_j-\tau_i} d(s)\, ds}x_j(t+\tau_j-\tau_i)\Big)x_i(t-\tau_i).\end{equation*} 
%
%

 From \cite{Faria23}, if
\begin{equation}\label{exti}p_i( y^0(t))<d(t+\tau_i)e^{\int^{t+\tau_i}_t d(r)\, dr},\q t\in [0,\omega],
\end{equation}
 for some $i\in\{ 1,\dots,n\}$,
where $y^0(t)$ is any positive solution of \eqref{y_ode},
then all nonnegative solutions $(S(t),x_1(t),\dots, x_n(t))$ of  \eqref{Chem}  satisfy $\lim_{t\to\infty}x_i(t)=0$.
In particular, if \eqref{exti} is satisfied for all $i=1,\dots,n$,
the washout solution $(y^*(t),0,\dots,0)$ 
 is a global attractor of all nonnegative solutions. In the case of periodic systems  \eqref{Chem}, this sufficient condition for extinction of all species  in the chemostat can be improved, as analysed below.

Firstly,  we remark that  Nah and R\"ost  \cite{NahRost} gave  a threshold criterion for stability of scalar linear $\omega$-periodic DDEs  of the form 
$x'(t)=-a(t)x(t)+b(t)x(t-\tau)$ which can be applied to the present situation. After normalising the delay, without loss of generality one may suppose that $\tau=1< \omega$ as in the statement below.

\begin{lem}\label{lem2.1} \cite{NahRost} Consider 
\begin{equation}\label{scalarRost}x'(t)=-a(t)x(t)+b(t)x(t-1),\end{equation} 
where $a(t),b(t)$ are continuous, nonnegative and $\omega$-periodic functions, for some $\omega>1$, 
 and define  $r=:\int_0^\omega (b(t)-a(t))\, dt$. If  $b(t+1)-a(t)$ does not change sign,  then: (i) if $r<0$, \eqref{scalarRost}
  is asymptotically stable; (ii) if $ r=0$, \eqref{scalarRost} is stable; (iii) if $  r>0$, \eqref{scalarRost} is unstable; in the latter case, solutions $x(t)$ of \eqref{scalarRost} with initial conditions $x_0=\var>0$ satisfy $\lim_{t\to\infty} x(t)=\infty$.
  \end{lem}

Together with comparison results for solutions of quasi-monotone DDEs, this lemma provides a  criterion for extinction of all microorganisms as follows.

 \begin{thm}\label{thm2.1} For \eqref{Chem}, assume that (H1)-(H2) hold.
 If 
 \begin{equation}\label{exti_peri1}p_i( y^*(t))\le d(t)e^{\int^{t+\tau_i}_t d(r)\, dr},\q t\in [0,\omega], i=1,\dots,n,\end{equation}
then the trivial solution $(y^*(t),0,\dots,0)$ of \eqref{Chem} is stable. If in addition 
 \begin{equation}\label{exti_peri2}\min_{t\in[0,\omega]}\Big(p_i( y^*(t))-d(t)e^{\int^{t+\tau_i}_t d(r)\, dr}\Big)<0, \q  i=1,\dots,n,\end{equation}
 then  $(y^*(t),0,\dots,0)$  is globally asymptotically stable (GAS).
% (ii)  ?? If there exists $i\in \{1,\dots,n\}$ such that
%  $$p_i( y^*(t))> d(t)e^{\int^{t+\tau_i}_t d(r)\, dr},\q t\in [0,\omega],$$
%  then the trivial periodic solution $(p^*(t),0,\dots,0)$ is unstable.
 \end{thm}
 
 \begin{proof} W.l.g. let $\omega >\tau_i$ for $i=1,\dots,n$. Write \eqref{Chem_y} as $x_i'(t)=g_i(t,x_t),\,  i=1,\dots,n$, where $g_i(t,x_t)=-d(t)x_i(t)+f_i(t,x_t)$. It is  clear that 
 \begin{equation}\label{linear_comp}
 \begin{split}
 g_i(t,\var)&\le -d(t)\var_i(0)+e^{-\int_{t-\tau_i}^td(s)\, ds}p_i(y^*(t-\tau_i))\var_i(-\tau_i)\\
 &=: h_i(t,\var),\q i=1,\dots,n,
 \end{split} \end{equation}
 for $t\ge 0$ and $\var=(\var_1,\dots,\var_n)\in  \mathfrak{C}^+$. For each $i\in\{1,\dots,n\}$, consider the linear periodic scalar DDE
\begin{equation}\label{lineari_per}
u_i'(t)=-d(t)u_i(t)+e^{-\int_{t-\tau_i}^td(s)\, ds}p_i(y^*(t-\tau_i))u_i(t-\tau_i)\end{equation}
Equations \eqref{lineari_per} satisfy the Smith's quasi-monotone condition (Q)  \cite[p.~78]{Smith}: for $t\ge 0, \var,\psi\in  \mathfrak{C}^+$, if $\var \le \psi$ and $\var_i(0)=\psi_i(0)$ holds for some $i$, then
$h_i(t,\var)\le h_i(t,\psi)$.
By a comparison result in \cite[Theorem 5.1.1]{Smith}, the solution of  \eqref{Chem_y} with an initial condition  $\var\in \mathfrak{C}^+$ at $t=0$ is bounded from above by the solution $(u_1(t),\dots,u_n(t))$ of system \eqref{lineari_per} ($1\le i\le n$) with the same initial condition. 
 
 Next, define $$r_i:=\int_0^\omega \Big (e^{-\int_{t-\tau_i}^td(s)\, ds}p_i( y^*(t-\tau_i))-d(t) \Big)\, dt,\q i=1,\dots,n.$$
Since $d(t)$ and $e^{-\int_{t-\tau_i}^td(s)\, ds}p_i(y^*(t-\tau_i))$ are $\omega$-periodic, $r_i$ is still given by
$$r_i=\int_0^\omega \Big (e^{-\int_t^{t+\tau_i}d(s)\, ds}p_i( y^*(t))-d(t) \Big)\, dt.$$
For each $i$, from hypothesis \eqref{exti_peri1} we derive that  the function under the integral does not change sign and $r_i\le 0$.
Lemma \ref{lem2.1} implies that \eqref{lineari_per}
is stable. Moreover, if \eqref{exti_peri1} and
 \eqref{exti_peri2} hold, then  $r_i<0$ and  \eqref{lineari_per} is asymptotically stable; since, the DDE is linear, it is also GAS. 
 %Moreover, if $p_i( y^*(t))\ge d(t)e^{\int^{t+\tau_i}_t d(r)\, dr},\ t\in [0,\omega]$ and $r_i>0$, the %equation \eqref{lineari_per} is unstable.
By the above comparison of solutions of \eqref{Chem_y} and \eqref{lineari_per} ($1\le i\le n$), the  result follows. \end{proof}

\begin{rmk}\label{rmk2.2} {\rm The comparison principle  \cite[Theorem 5.1.1]{Smith} gives a comparison of solutions for two different DDEs $x'(t)=f_1(t,x_t), x'(t)=f_2(t,x_t)$ if $f_1(t,\var)\le f_2(t,\var)$ on a set $\Omega\subset dom\, f_1=dom\, f_2$ and either $f_1$ or $f_2$ satisfies the quasi-monotone condition (Q). Note nevertheless that, once the existence and continuous dependence of solutions on initial data are established, the proof of  \cite[Theorem 5.1.1]{Smith}  applies to {\it mixed-type} FDEs as above.  
}\end{rmk}
\begin{rmk}\label{rmk2.3} {\rm For  \eqref{scalarRost} with $a(t),b(t)$ $\omega$-periodic continuous functions ($\omega>1$),  if $b(t+1)-a(t)$ changes its sign,   examples were presented in  \cite{NahRost} showing  that, in general,  $r=:\int_0^\omega (b(t)-a(t))\, dt=0$ is not a stability threshold.  
On the other hand, if $p_i$ are $C^1$-smooth functions, system \eqref{lineari_per} ($1\le i\le n$) is the linearisation of \eqref{Chem_y} at 0. Even for periodic ODE systems with the trivial solution 0, it is known that the stability of its linearisation at 0 does not imply, in general, the stability of 0 as a solution of the original system, much less its global attractivity (see e.g. \cite[p.~87]{Bellman}). Nevertheless, 
in virtue of the inequalities \eqref{linear_comp}, %and comparison result for functional differential equations satisfying (Q)
an interesting open problem is whether comparison result for FDEssatisfying (Q)
 can lead to the global attractivity of the trivial washout solution $(y^*(t),0,\dots,0)$ of \eqref{Chem} assuming simply the inequalities \eqref{exti_peri1}. In fact, this is true for the associated ODE chemostat model \eqref{Chem}, where all the delays $\tau_i$ are zero, if $p_i$ are strictly increasing: it was established in \cite{WZ} that $\int_0^\omega (p_i( y^*(t))-d(t))\, dt\le 0\ (1\le i\le n)$ are sufficient conditions for the extinction of all microorganisms.}\end{rmk}

\section{Existence of nontrivial periodic solutions}
\setcounter{equation}{0}

The existence of nontrivial $\omega$-periodic solutions of \eqref{Chem} is now addressed. Beforehand, further notation is introduced and an auxiliary result given.

For $n\in\N$ fixed, denote by $C_\omega:=C_\omega(\R,\R^n)$ the space of  continuous and $\omega$-periodic functions $\var:\R\to\R^n$ with the supremum norm, and by  $C_\omega^+:=C_\omega^+(\R,\R^n)$ the cone of its nonnegative elements.
  %, and the order in  $PC_\omega(\R,\R^n)$ induced by $PC_\omega^+(\R,\R^n)$.   
 For $n=1$, $C_\omega(\R)$ and $C_\omega^+(\R)$ are used with the obvious meaning. 
  As before, $\R^n$ is also  seen as the set of constant functions, thus $\R^n\subset C_\omega$.
%  The space $C_\omega$ can be naturally identified as a subset of both $C([-\tau,0];\R^n)$ and $\mathfrak{C}=C([-\tau,\tau^0];\R^n)$ by the injection $x\mapsto x_0$ (since, without loss of generality, one may consider $\tau\ge \omega$). 
 
 Fix the maximum norm in $\R^n$. For $x(t)=(x_1(t),\dots,x_n(t))\in C_\omega^+$, we have
$\dps\|x\|=\max_{1\le i\le n}\max_{t\in[0,\omega]}x_i(t)$, and define
$$\ul{x}=\min_{1\le i\le n}\min_{t\in[0,\omega]}x_i(t).$$
For a suitable  $\sigma\in(0,1)$,   consider a new cone $K(\sigma)$ in $C_\omega^+$ given by
\begin{equation}\label{K}
K({\sigma}):=\{x\in C_\omega^+:x_i(t)\geq\sigma \|x_i\|, t\in [0,\omega], i=1,\dots,n\}.
\end{equation}
If  $\sigma\in(0,1)$ is fixed, we simply write  $K$ instead of $K({\sigma})$.

Since $d(t),y^*(t)$ are $\omega$-periodic functions,  for $x\in C_\omega^+$ it is clear that the functions
$$t\mapsto \int_{t}^{t+\omega}e^{\int _t^{s-\tau_i} d(r)\, dr}p_i^*(s,x_s)\, ds,$$
where 
\begin{equation}\label{pi*}p_i^*(t,x_t)=p_i\bigg(y^*(t-\tau_i)-\sum_{j=1}^n e^{\int_{t-\tau_i}^{t+\tau_j-\tau_i} d(s)\, ds}x_j(t+\tau_j-\tau_i)\bigg),\end{equation}
are $\omega$-periodic and nonnegative as well $(1\le i\le n)$. Consider the extension of each  function $p_i(x)$ given by $\tilde p_i(x)=\begin{cases} p_i(x)&{\rm if} \ x\ge 0\\ 0&{\rm if}\  x<0\end{cases}.$  In what follows,  as in the definition of $\Phi$ below, for simplicity  we drop the tilde in $\tilde p_i(x)$ and  write $p_i(x)$ instead of $\tilde p_i(x)$.  

Define the operator $\Phi:C_\omega^+\to C_\omega^+$  by
\begin{equation}\label{Phi}
\begin{split}
 \Phi &=(\Phi_1,\dots, \Phi_n),\\
(\Phi_i x)(t)&=\big(e^{D(\omega)}-1\big)^{-1}\int_t^{t+\omega}e^{\int_t^sd(r)\, dr}f_i(s,x_{s})\, ds,\, \, t\in\R,
\end{split}
\end{equation}
where  $D(\omega):=\int_0^\omega d(s)\, ds$  and $f_i$ as in \eqref{f_i} are rewritten as
\begin{equation}\label{f_Chem_y}
\begin{split}
% f_0(t,\var) &=f_0(t)=d(t)S_0(t),\\
f_i(t,x_t)&=e^{-\int_{t-\tau_i}^td(s)\, ds}p_i^*(t,x_t)x_i(t-\tau_i),\ i=1,\dots,n,\end{split}
\end{equation}
for $p_i^*(t,x_t)$  in \eqref{pi*} and $x(t)=(x_1(t),\dots,x_n(t))$.

\begin{lem}\label{lem3.1} With the above notations, take $\sigma\le e^{-D(\omega)}$  and $K=K(\sigma)$. Then:\vskip 0mm
(i) $x$ is a {fixed point} of   $\Phi$ if and only if $x$ is a nonnegative $\omega$-periodic solution of \eqref{Chem_y}; \vskip 0mm
 (ii) $\Phi(K)\subset K$;\vskip 0mm
 (iii) $\Phi$ is continuous and  compact on $K$. \end{lem}

\begin{proof} (i)  We have $\frac{d}{dt} (\Phi_i x)(t)=-d(t)(\Phi_i x)(t)+f_i(t,x_t)\, (1\le i\le n)$, thus a fixed point of $\Phi$ is an $\omega$-periodic solution of \eqref{Chem_y}. Conversely, if $x\in C_\omega^+$ is a solution of \eqref{Chem_y}, then 
\begin{equation*}
\begin{split}(\Phi_i x)(t)&=\big(e^{D(\omega)}-1\big)^{-1}\int_t^{t+\omega}\!\! e^{\int_t^sd(r)\, dr}(x_i'(s)+d(s)x_i(s))\, ds\\
&= \big(e^{D(\omega)}-1\big)^{-1}\Big[x_i(s)e^{\int_t^sd(r)\, dr}\Big]_t^{t+\omega}=x_i(t).\end{split}
\end{equation*}

(ii) For 
$$C_i(x):=\big(e^{D(\omega)}-1\big)^{-1}\int_0^{\omega}f_i(s,x_{s})\, ds,\q x\in C_\omega^+,$$
we have $(\Phi_i x)(t)\le e^{D(\omega)}C_i(x)$ and $(\Phi_i x)(t)\ge C_i(x)$ for all $i$ and $t\in [0,\omega]$, thus it follows that $\Phi(K)\subset K$ if $0<\sigma\le e^{-D(\omega)}$, for $K=K({\sigma})$.

(iii) Clearly $\Phi$ is continuous. For $f_i$ given by \eqref{f_Chem_y}, denote $f=(f_1,\dots,f_n)$. The function $F(t,x):=f(t,x_t)$ for $ t\in\R, x\in C_\omega^+:=C_\omega^+ (\R;\R^n)$
 is uniformly equicontinuous  on $t\in [0,\omega]$ on  bounded sets of $C_\omega^+$, i.e., for any $A\subset C_\omega^+$ bounded and $\vare>0$, there is $\de>0$ such that $\max_{t\in [0,\omega]}|F(t,x)-F(t,y)|<\vare$ for all  $x,y\in A$ with $\|x-y\|<\de$. By using Ascoli-Arzel\`a theorem,  one easily sees that $\Phi$ transforms bounded sets of $K$ on relative compact sets.
%(...) See  \cite{FF} for more details and complete proofs of some of the assertions above.
\end{proof}

 Next, we look for  a fixed point of   $\Phi_{|_K}:K\to K$ by employing the Krasnoselskii  fixed point theorem in the version given below (see \cite[Theorems   7.3 and 7.6]{AgMeeRegan}).%-- which includes both the  compressive and expansive forms --  .
% of Krasnoselskii theorem are  often used , which  can be found in   \cite[Section 20]{Deimling}. %\cite[Chapter 7]{AMO}.

%\begin{thm}\label{thmKras} \cite{GuoLaks} Let $X$ be a Banach space, $K$ a  cone in $X$,  and  
% $A_{r,R}:=\{y\in K:r\leq\Vert y\Vert \leq R\}$, for some $0<r<R$.
% Let $T:A_{r,R}\to K$ be a completely continuous operator. Suppose that there exist $r,R$ with $0<r<R$ such that one of the following forms is satisfied:
% 
% (a) compressive form:
%		$\|T y\|\leq R$ if $y\in K, \|y\|=R$ and $\|T y\|\ge r$ if $y\in K, \|y\|=r$;
%		
%(b) expansive form: 
%		 $\|T y\|\leq r$ if $y\in K, \|y\|=r$ and  $\|T y\|\ge R$ if $y\in K, \|y\|=R$.\\		 
%		 Then there exists a fixed point $y^*$ of  $T$ in  $A_{r,R}$.
%\end{thm}	
%
%
%OR:

%\cite[Theorem 20.1,p.~239]{Deimling}?? CONF the version below in \cite{AgMeeRegan}.

\begin{thm}\label{thmKras}\cite{AgMeeRegan} Let $K$ be a closed cone in a Banach space, $r,R\in\R^+$ with $r\ne R$, $r_0=\min\{r,R\} , R_0= \max\{r,R\}$ and   $K_{r_0,R_0}:=\{x\in K:r_0\leq\|x\| \leq R_0\}$.
 %$K_{R_0}:=\{x\in K:\|x\| \leq R_0\}$. 
 Assume that
 $T:K_{r_0,R_0} \longrightarrow K$ is a completely continuous operator such that
\begin{enumerate}
\item[(i)]% $\|Tx\|\le R$ for all $x \in K$ with $\|x\|=R$;
$Tx\ne \la x$ for all $x \in K$ with $\|x\|=R$ and all $\la>1$;
\item [(ii)] There exists $\psi \in K\setminus\{0\}$ such that $x \neq Tx + \lambda \psi$ for all $x \in K$ with $\|x\|=r$ and all $\lambda >0$.
\end{enumerate}
Then $T$ has a fixed point  in $K_{r_0,R_0}$.
\end{thm}

The above theorem includes both the compressive (with $r<R$) and expansive (with $r>R$) forms of Krasnoselskii's  fixed point theorem. 
To apply the  compressive form of Theorem \ref{thmKras} to \eqref{Chem_y} in a cone $K=K(\sigma)$, we  must impose an assumption assuring a   kind of  {\it ``superlinear"} growth at 0. As it will become clear,  a {\it ``sublinear"} growth at $\infty$ is a priori guaranteed by the dissipativeness of the system.

\begin{thm}\label{thm3.3}Consider \eqref{Chem} with $S_0(t),d(t)$ satisfying (H1)-(H2)  and assume  the following condition:
\begin{itemize}
\item[(H3)]    the unique $\omega$-periodic solution $y^*(t)$ of \eqref{y_ode} satisfies
$$\min_{t\in[0,\omega]}\int_{t}^{t+\omega}e^{\int _{t+\tau_i}^{s} d(r)\, dr}p_i(y^*(s))\, ds>e^{\int_0^\omega d(r)\, dr}-1,\q i=1,\dots,n.$$\end{itemize}
Then there exists a  nonnegative $\omega$-periodic solution $(S^*(t),x_1^*(t),\dots, x_n^*(t))$ of  \eqref{Chem}. Moreover, $0<S^*(t)<y^*(t)$, there exists $i\in \{1,\dots,n\}$ such that $x_i^*(t)>0$ for $ t\in\R$ and
$$ \min_{t\in [0,\omega]}x_j^*(t)\ge e^{-\int_0^\omega d(r)\, dr}\max_{t\in [0,\omega]}x_j^*(t),\q j=1,\dots,n.$$
\end{thm}

\begin{proof} We  show that $\Phi_{|_K}:K\to K$ satisfies (i) and (ii) in Theorem \ref{thmKras} by using  arguments similar to the ones in  \cite[Theorem 3.1]{FF}. 
%For completeness of the reader, we include the proof here.

%The  functions $p_i^*(t,x_t)$ in \eqref{pi*} are bounded along solutions $x(t)\in C_\omega^+$, $p_i^*(t,x_t)\le p_i(y^*(t-\tau_i))\le p_i\big(\ol{S^0}\big)$. Thus, we may consider the operator $\Phi$ defined in \eqref{Phi}, \eqref{f_Chem_y} with $p_i(x)$ in \eqref{pi*} replaced by $\hat p_i(x)=\begin{cases} p_i(x)\ {\rm if} \ 0\le x\le \ol{S^0}\\ p_i\big(\ol{S^0}\big)\ {\rm if}\  x\ge \ol{S^0}.\end{cases}$ For simplicity, we drop the hats in $\hat p_i(x)$.
%
%For $M>e^{-\int_{t-\tau_i}^td(s)\, ds}p_i\big(\ol{S^0}\big)$
%
%
%Since system \eqref{Chem} is dissipative, the functions $f_i(t,x_t)$ in \eqref{f_Chem_y} are bounded along solutions $x(t)\in C_\omega^+$, i.e., $f_i(t,x_t)\le M$ for some $M>0$ and all solutions $x(t)\ge 0$. 
%
%In this way, $f_i(t,x_t)\le \vare \|x\|$ for $\|x\|\ge \vare^{-1}M$ for any $\vare>0$.

For $x\in K$ with $\|x\|=R$, recall that  $\sigma  \|x_i\|\le x_i(t)\le R$ for all $i$, $t\in [0,\omega]$,  and that $ \|x_j\|=R$ for  some $j\in \{1,\dots,n\}$. Thus, we have $p_i^*(t,x_t)\le  p_i\Big (y^*(t-\tau_i)-\sigma \sum_{j=1}^n e^{\ul{d}\tau_j}\|x_j\|\Big)=0$ for $R>0$  chosen sufficiently large  so that $\sigma e^{\ul{d}\min_j\tau_j} R\ge \ol{S_0}$.
%Since system \eqref{Chem} is dissipative, the functions $f_i(t,x_t)$ in \eqref{f_Chem_y} are bounded along solutions $x(t)\in C_\omega^+$, i.e., $f_i(t,x_t)\le M$ for some $M>0$ and all solutions $x(t)\ge 0$. 
In this way, if  $x\in K$ with  $\|x\|=R$ we have   $f_i(t,x_t)\equiv 0$ for all $i$, and  in particular condition (i) in Theorem \ref{thmKras} is fulfilled.
%we conclude that $\Phi x\ne \la x$ for all $\la >1$, $x\in K$ with $\|x\|=R$, hence condition (i) in Theorem \ref{thmKras} is fulfilled.
%$f_i(t,x_t)\le \vare \|x\|$.

%\
%
%
%For $x\in K$ with $\|x\|=\vare^{-1}M=:R$, 
%	let $i$ be such that $\|x\|=\|x_i\|=R$. For such $i$, we have $x_i(t)\le R$ and $x_i(t)\ge \sigma  R$ for $t\in [0,\omega]$. 	
	
%	Using the definition of $\Phi$, we have
% \begin{equation}\label{3.6}
% \begin{split}\|\Phi_i x\|&=\big(e^{D(\omega)}-1\big)^{-1}\max_{t\in[0,\omega]}\int_t^{t+\omega} e^{\int _t^sd(r)\, dr}f_i^*(s,x_s)\, ds\\
% &\le \vare R\big(e^{D(\omega)}-1\big)^{-1}\max_{t\in[0,\omega]}\int_t^{t+\omega} e^{\int _t^sd(r)\, dr}ds\le 1
% \end{split}\end{equation}
%for $\vare>0$ sufficiently small. In particular, we conclude that $\Phi x\ne \la x$ for all $\la >1$, $x\in K$ with $\|x\|=R$, hence condition (i) in Theorem \ref{thmKras} is fulfilled.
%moreover,  for $C_i^\infty$ as in  \eqref{sublinear_infty} we have $C_i^\infty\le 1$ for $\vare$ sufficiently small. It remains to verify that condition (H6)(a) (or (H6')(a)) is satisfied.
%
%Consider  the cone $K=\{x\in C_\omega^+: x_i(t)\ge \sigma \|x_i\|, t\in [0,\omega], i=1,\dots,n\}$.
 %Since $p_i(x)$ are  non-decreasing functions,
 
 Next, we show that (H3) implies condition (ii) in Theorem \ref{thmKras}, for some $r>0$ sufficiently small. Denote $\ul{x_i}=\min_{t\in[0,\omega]} x_i(t)$.
 For $f_i \, (1\le i\le n)$ as in \eqref{f_Chem_y} and $x\in C_\omega^+$, $t\in\R$, observe that
 \begin{equation}\label{f_yy}
\begin{split}
f_i(t,x_t)&\ge e^{-\int_{t-\tau_i}^td(s)\, ds}p_i^*(t,x_t)\ul{x_i},\\
&\ge e^{-\int_{t-\tau_i}^td(s)\, ds}p_i\Big(y^*(t-\tau_i)-\sum_{j=1}^n e^{\ol{d}\tau_j}\|x\|\Big)\ul{x_i},\ \  i=1,\dots,n,  t\in\R.
\end{split}
\end{equation}
Clearly, the function % $d(t),y^*(t)$ are $\omega$-periodic functions, it is clear that 
$F_i(t):=\int_{t}^{t+\omega}e^{\int _t^{s-\tau_i} d(r)\, dr}p_i(y^*(s-\tau_i))\, ds$
is $\omega$-periodic, thus $\min_{t\in[0,\omega]}F_i(t)=\min_{t\in[0,\omega]}F_i(t+\tau_i)$, i.e.,
\begin{equation*}
\begin{split}\min_{t\in[0,\omega]}\int_{t}^{t+\omega}e^{\int _t^{s-\tau_i} d(r)\, dr}p_i(y^*(s-\tau_i))\, ds%&=\min_{t\in[0,\omega]}\int_{t-\tau_i}^{t+\omega-\tau_i}e^{\int _t^{s} d(r)\, dr}p_i(y^*(s))\, ds\\
&=\min_{t\in[0,\omega]}\int_{t}^{t+\omega}e^{\int _{t+\tau_i}^{s} d(r)\, dr}p_i(y^*(s))\, ds,\end{split}\end{equation*}
for $i=1,\dots,n.$ From (H3), for all $i$
\begin{equation}\label{estimateFi}F_i(t)>e^{\int_0^\omega d(r)\, dr}-1\q {\rm for}\q t\in [0,\omega].\end{equation}
Fix $\de>0$ small enough so that $\min_{t\in[0,\omega]}F_i(t)>e^{\int_0^\omega d(r)\, dr}-1+\de \frac{e^{\ol{d}\omega}-1}{\ol{d}}$.
On the other hand, from the uniform continuity of $p_i(x)$ on $[\ul{S_0},\ol{S_0}]$, there is $\vare>0$ such that $p_i(x)-p_i(x-\vare)<\de$ for all $x\in[\ul{S_0},\ol{S_0}]$, and
\eqref{estimateFi} therefore yields
\begin{equation}\label{newh3}\min_{t\in[0,\omega]}\int_t^{t+\omega }e^{\int _t^{s-\tau_i} d(r)\, dr}p_i(y^*(s-\tau_i)-\vare)\, ds>e^{\int_0^\omega d(r)\, dr}-1,\end{equation}
for $i=1,\dots,n$.
Choose %$b_{i}(t)=e^{-\int_{t-\tau_i}^td(s)\, ds}p_i\big(y^*(t-\tau_i)-\vare\big)$, and 
  $r=\frac{1}{n}\vare  e^{-\ol{d}\tau}$, where as before $\tau=\max_{1\le i\le n}\tau_i$. For $x\in K$ with
 $\|x\|\le r$, we have $\sum_{j=1}^n e^{\ol{d}\tau_j}\|x\|\le \vare$. Together with  \eqref{f_yy} this leads to
%$$f_i(t,x_t)\ge b_{i}(t)\ul{x_i}, \q i=1,\dots,n,  t\in\R.$$
% %For the constants $c_i^0$ in \eqref{sublinear0},
% From \eqref{newh3}, we now obtain
 \[
(\Phi_i x)(t)\ge 
c_i  \ul{x_i}\] where
\begin{equation}\label{sublinear0_Chem0}
\begin{split}
c_i:=\big(e^{\int_0^\omega d(r)\, dr}-1\big)^{-1} \min_{t\in[0,\omega]}\int_t^{t+\omega}e^{\int _t^{s-\tau_i} d(r)\, dr} p_i\big(y^*(s-\tau_i)-\vare\big)\, ds>1, \  i=1,\dots,n.
\end{split}\end{equation}

Suppose now that there are $\la>0$ and   $x\in K$ with $\|x\|=r$, such that 
 $x=\Phi x+\la {\bf 1}$, where ${\bf 1}$ is interpreted as the constant function $(1,\dots,1)$, so that $x_i=\Phi_ix+\la,\, i=1,\dots,n$.
		%In particular, $y\ge \la>0$, thus $ y\in int(X^+)$. 
		Let $\mu:=\min_{t\in [0,\omega]}\min_{1\le i\le n}x_i(t)$. Observe that $0<\la\le \mu \le x_i(t)\le r$  for $t\in [0,\omega], i=1,\dots,n$. From \eqref{sublinear0_Chem0} we have $c_i>1$, which   yields 
$$(\Phi_i x)(t)\ge \mu c_i> \mu,\qq t\in [0,\omega], i=1,\dots,n.$$
% \begin{equation}\label{3.7}(\Phi_i x)(t)\ge \mu\,\big(e^{\int_0^\omega d(r)\, dr}-1\big)^{-1}\,\min_{t\in[0,\omega]}\int_t^{t+\omega} e^{\int _t^sd_i(r)\, dr}\Big (\sum_{j\ne i}a_{ij}(s)+b_{1i}(s)\Big )\, ds\ge\mu.\end{equation}
Choose 
		$t^*\in [0,\omega]$ and $i^*\in \{1,\dots,n\}$ such that $x_{i^*}(t^*)< \mu+\la$. We obtain
\begin{equation*}
\begin{split}
\mu&> x_{i^*}(t^*)-\la=(\Phi_{i^*} x)(t^*)>\mu,
 \end{split}
\end{equation*}	
which is a contradiction.	Therefore, Theorem \ref{thmKras} guarantees that $\Phi$  admits a fixed point $x^*$  in $K_{r,R}=\{x\in K:r\leq\|x\| \leq R\}$, which is
an $\omega$-periodic solution   $x^*(t):=(x_1^*(t),\dots,x_n^*(t))$ of \eqref{Chem_y} in $K\setminus \{0\}$. This means that  \eqref{Chem} admits an $\omega$-periodic  solution $(S^*(t),x_1^*(t),\dots, x_n^*(t))$ with  the first component $S^*(t)$ given by the relation \eqref{S(t)}. %with $S^*(t)<y^*(t),\, t\in\R$.

%Next, 
%assume (H3). There is $r_0>0$ such that $p_i(x)>\max_{t\in [0,\omega]} \bigg (d(t)e^{\int_{t-\tau_i}^t d(r)\, dr}\bigg)$ for $x\in [S_0-r_0,S_0]$ and  each $i\in \{1,\dots,n\}$, hence
%\begin{equation*}
%\begin{split}
%c_i^0
%&>
%\big(e^{\int_0^\omega d(r)\, dr}-1\big)^{-1}\, \min_{t\in[0,\omega]}\int_t^{t+\omega}e^{\int _t^s d(r)\, dr}d(s)\, ds=1, \q i=1,\dots,n,
%\end{split}\end{equation*}
%Again, Theorem \ref{thmP.1} guarantees the existence  of an $\omega$-periodic solution  $x^*(t)$.
%

Next, from \cite{Faria23}, $S^*(t)>0$ for $t\in [0,\omega]$. 
Suppose that $S^*(t_1)=y^*(t_1)$ for some $t_1$. Thus $x_i^*(t_1)=0$ for all $i$ 
and the definition of $K$ leads to $x_i^*(t)\equiv 0$ for $i=1,\dots,n$,  therefore $x^*(t)\equiv 0$. But this contradicts the fact that $x^*\in K\setminus \{0\}$.  Hence, $0<S^*(t)< y^*(t)$ for all $t\in [0,\omega]$,  and $\|x_i^*\| >0$ for some $i\in \{1,\dots,n\};$ since $x^*\in K$, again we obtain that $x_i(t)$ is positive on $[0,\omega]$.
\end{proof}

Several immediate consequences of this theorem are drawn below. To start with, instead of (H3) one may impose either a pointwise or an average comparison condition between $p_i(y^*(t))$ and $d(t)$, which are more restrictive but  easier to verify than (H3).

\begin{thm}\label{thm3.4}For \eqref{Chem} with  $d,S_0\in C_\omega(\R)$ and positive, assume  (H2) and that  the unique $\omega$-periodic solution   $y^*(t)$  of \eqref{y_ode} satisfies one of the following conditions:
\begin{itemize}
\item[(H3.a)] $\dps\min_{t\in [0,\omega]} \Big (p_i( y^*(t))-d(t+\tau_i)e^{\int_t^{t+\tau_i} d(r)\, dr}\Big)>0,\ i=1,\dots,n;$
\item[(H3.b)] $\dps \int_0^\omega p_i( y^*(t))\, dt \ge \Big (e^{\int_0^\omega d(r)\, dr}-1\Big) \max_{t\in [0,\omega]}e^{\int_{t}^{t+\tau_i} d(r)\, dr},\ i=1,\dots,n.$
%\item[(H3.a*)] (TIRAR!!)
%$\dps \min_{t\in [0,\omega]} \frac{p_i( y^*(t))}{d(t)}> \max_{t\in [0,\omega]}e^{\int_{t}^{t+\tau_i} d(r)\, dr}, \, i=1,\dots,n;$
\end{itemize}
Then there exists (at least) one  nonnegative $\omega$-periodic solution $(S^*(t),x_1^*(t),\dots, x_n^*(t))$ of  \eqref{Chem}, with $0<S^*(t)<y^*(t)$ and $x_i^*(t)>0,\ t\in\R$,
 for some $i\in \{1,\dots,n\}$. 
\end{thm}

\begin{proof}  It suffices to show that each one of conditions (H3.a), (H3.b) implies (H3). If (H3.a) holds, for  $t\in [0,\omega], i=1,\dots,n$, we obtain
$$\int_{t}^{t+\omega}e^{\int _{t+\tau_i}^{s} d(r)\, dr}p_i(y^*(s))\, ds>\int_{t}^{t+\omega}e^{\int _{t+\tau_i}^{s+\tau_i} d(r)\, dr}d(s+\tau_i)\, ds
=e^{\int_0^\omega d(r)\, dr}-1.$$
Note also that
$$\int_{t}^{t+\omega}e^{\int _{t+\tau_i}^{s} d(r)\, dr}p_i(y^*(s))\, ds> e^{-\int_{t}^{t+\tau_i} d(r)\, dr}\int_0^\omega p_i( y^*(s))\, ds.$$
%(TIRAR) Next, assuming (H3.a*), for  $t\in [0,\omega], i=1,\dots,n$, we have
%\begin{equation*}
%\begin{split}\int_{t}^{t+\omega}e^{\int _{t+\tau_i}^{s} d(r)\, dr}p_i(y^*(s))\, ds&=e^{-\int_{t}^{t+\tau_i} d(r)\, dr}\int_{t}^{t+\omega}e^{\int _{t}^{s} d(r)\, dr}p_i(y^*(s))\, ds\\
%&> \int_{t}^{t+\omega}e^{\int _t^{s} d(r)\, dr}d(s)\, ds 
%=e^{\int_0^\omega d(r)\, dr}-1.\end{split}\end{equation*}
This establishes that  (H3) is fulfilled if either (H3.a) or (H3.b) holds.
\end{proof}

 \begin{rmk}\label{rmk3.1} 
  %For model \eqref{Chem0} with $d(t)$ periodic and $p_i(x)$ as in (H2), if 
{\rm The reversed inequalities in (H3.a), i.e., for  $p_i(y^*(t))< d(t+\tau_i)e^{\int_{t}^{t+\tau_i}d(r)\, dr},\, t\in [0,\omega],  i=1,\dots,n$,  
 guarantee that the washout solution $(y^*(t),0,\dots, 0)$ is globally attractive \cite{Faria23}. On the other hand, 
 %although $ \min_{t\in [0,\omega]} \frac{p_i( y^*(t))}{d(t)}> \max_{t\in [0,\omega]}e^{\int_{t}^{t+\tau_i} d(r)\, dr} \,(1\le i\le n)$ implies (H3), 
  $p_i( y^*(t))> d(t)e^{\int^{t+\tau_i}_t d(r)\, dr},\, t\in [0,\omega]\,(1\le i\le n)$ a priori  is not a sufficient condition to derive that \eqref{Chem} possesses  a nonnegative and nontrivial $\omega$-periodic solution. In any case, one can easily verify that
  \begin{itemize}
\item[(H3.c)] 
$\dps \min_{t\in [0,\omega]} \frac{p_i( y^*(t))}{d(t)}> \max_{t\in [0,\omega]}e^{\int_{t}^{t+\tau_i} d(r)\, dr}, \, i=1,\dots,n,$
\end{itemize}
implies (H3), and consequently the conclusion in Theorem \ref{thm3.3} follows.
 Further comments are presented  in Remark \ref{rmk2.3}.} \end{rmk}
  
  In what follows, ${\cal M}(a)$ denotes the mean value of an $\omega$-periodic continuous function $a(t)$, ${\cal M}(a)=\omega^{-1}\int_0^{\omega }a(t)\, dt$.

\begin{rmk}\label{rmk3.2} {\rm 
%Denote  ${\cal M}(a)$ as the mean value of an $\omega$-periodic continuous function $a(t)$, ${\cal M}(a)=\omega^{-1}\int_0^{\omega }a(t)\, dt$. 
Consider the mean value of $p_i(y^*(t))$,
${\cal M}(p_i(y^*))=\omega^{-1}\int_0^{\omega }p_i(y^*(t))\, dt$. If ${\cal M}(p_i(y^*)) >\omega^{-1} (e^{\int_0^\omega d(r)\, dr}-1)$ for $ i=1,\dots,n$ and the delays $\tau_i$ are sufficiently small, then  (H3.b) is satisfied. For the competitive chemostat ODE model obtained from \eqref{Chem}  by taking all the delays equal to zero, condition ${\cal M}(p_i(y^*))>{\cal M}(d)$ guarantees that a nonnegative, nontrivial periodic solution exists, see\cite{WZ,Zhaobook}.
 Clearly, condition ${\cal M}(p_i(y^*))>{\cal M}(d)$ is
 less restrictive than ${\cal M}(p_i(y^*)) >\omega^{-1} (e^{\int_0^\omega d(r)\, dr}-1)$. However,  from the   above criteria ${\cal M}(p_i(y^*))>{\cal M}(d)$ is not sufficient to deduce the existence of such a  nontrivial periodic solution for \eqref{Chem}. 
 %In the case of no delays ($\tau_i=0$ for all $i$), 
  For small delays $\tau_i<\omega$, it is therefore natural to inquire whether the above result can be   deduced with (H3.a) or (H3.b) replaced by the weaker assumption ${\cal M}(p_i(y^*))>{\cal M}(G_i)$, for
 $G_i(t)= d(t+\tau_i)e^{\int^{t+\tau_i}_t d(r)\, dr}$.}
  \end{rmk}  

The above proofs also show that if either assumption (H3) or one of its variants (H3.a) or (H3.b)   holds  for a fixed $i\in\{ 1,\dots,n\}$, then  \eqref{Chem} possesses an $\omega$-periodic solution of the form  $(S_i^*(t),0,\dots, x_i^*(t), 0,\dots, 0)$ with $x_i^*(t)>0$. To deduce this, it is sufficient to restrict the analysis to the planar   system
\begin{equation}\label{Chemi}
 \begin{split}
 S'(t)&=d(t)(S_0(t)-S(t))-\sum_{i=1}^np_i(S(t))x_i(t)\\
x_i'(t)&=-d(t)x_i(t)+e^{-\int_{t-\tau_i}^td(s)\, ds}p_i(S(t-\tau_i))x_i(t-\tau_i),\\
\end{split}
\end{equation}
for each $i\in\{1,\dots,n\}$ fixed.
 Thus, Theorems \ref{thm3.3} and \ref{thm3.4} further provide criteria  for  the  existence of $n$  nonnegative and nontrivial $\omega$-periodic solutions.

\begin{cor}\label{cor3.1} Under the conditions of Theorem \ref{thm3.3} or Theorem \ref{thm3.4}, system \eqref{Chem} has  (at least) $n$  nonnegative $\omega$-periodic solutions of the form $(S_i^*(t),0,\dots, x_i^*(t), 0,\dots, 0)$ $(1\le i\le n)$ with $0<S_i^*(t)<y^*(t)$ and $x_i^*(t)>0$, $t\in\R$.
\end{cor}

 For a chemostat model with a single micro-organism, we obtain:

\begin{cor}\label{cor3.3} Consider the chemostat model
 \begin{equation}\label{Chem_one}
 \begin{split}
 S'(t)&=d(t)(S_0(t)-S(t))-\sum_{i=1}^n p(S(t))x(t)\\
x'(t)&=-d(t)x(t)+e^{-\int_{t-\tau}^td(s)\, ds}p(S(t-\tau))x(t-\tau),
\end{split}
\end{equation}
with $d(t),S_0(t)\in C_\omega ^+(\R)$ and positive, $\tau>0$ and $p(x)$ continuous, non-decreasing with $p(0)=0, p(x)>0$ for $x>0$. For    $y^*(t)$  the unique $\omega$-periodic solution of \eqref{y_ode}, suppose that
$$\min_{t\in[0,\omega]}\int_{t}^{t+\omega}e^{\int _{t+\tau}^{s} d(r)\, dr}p(y^*(s))\, ds>e^{\int_0^\omega d(r)\, dr}-1.$$
Then there is (at least) one  positive $\omega$-periodic solution $(S^*(t),x^*(t))$. In particular, this is the case if either
$$p( y^*(t))>d(t+\tau)e^{\int^{t+\tau}_t d(r)\, dr},\q t\in [0,\omega],$$
or $$ \int_0^\omega p( y^*(t))\, dt \ge \Big (e^{\int_0^\omega d(r)\, dr}-1\Big) \max_{t\in [0,\omega]}e^{\int_{t}^{t+\tau} d(r)\, dr}.$$
\end{cor}

 \begin{rmk}\label{rmk3.3} {\rm The above corollary  improves a result in \cite{ARS20}, where the authors proved the existence of a positive periodic solution for \eqref{Chem_one} under more restrictive assumptions: in  \cite{ARS20} the response function $p(x)$  is assumed to be $C^1$-smooth with $p'(x)$ positive and bounded, the delay satisfies $\tau<\omega$ and the additional condition $p( y_\la^*(t))>d(t)e^{\int^{t+\tau}_t d(r)\, dr}$  for all $ t\in [0,\omega]$ and all $\la\in (0,1]$ is imposed, where $y_\la^*(t)$ is the positive $\omega$-periodic solution of the ODEs $y'=\la d(t)\big (S_0(t)-y\big)$, for  $\la\in (0,1]$.}
\end{rmk}

In the case of a constant initial concentration $S_0(t)\equiv S_0$,  the positive  periodic solution of \eqref{y_ode} is $y^*(t)\equiv S_0$, as in the following corollary.

% in which case one may consider the equivalent form \eqref{Chem1} for $u(t)=S_0-S(t)$. 
\begin{cor}\label{cor3.2}Consider the chemostat system \eqref{Chem0} with $S_0>0$ and  $d(t)\in C_\omega(\R)$ and positive. Assume (H2) and that  the following condition is satisfied:
\begin{equation}\label{H4}
\dps p_i(S_0)\, \min_{t\in[0,\omega]}\int_t^{t+\omega}e^{\int _t^{s-\tau_i} d(r)\, dr} \, ds> e^{\int_0^\omega d(r)\, dr}-1\q {\rm for}\q i=1,\dots,n.
%\item[(H3)]  $\dps p_i(S_0)>\max_{t\in [0,\omega]} \bigg (d(t)e^{\int_{t-\tau_i}^t d(r)\, dr}\bigg)$ for $i=1,\dots,n$.
\end{equation}
Then there exists (at least) $n$  nonnegative $\omega$-periodic solution $(S_i^*(t),0,\dots, x_i^*(t), 0,\dots, 0)$ $(1\le i\le n)$ with $0<S^*(t)<S_0$ and $x_i^*(t)>0,\ t\in\R$,
 for some $i\in \{1,\dots,n\}$. In particular this is the case if \begin{equation}\label{h4.0}\dps p_i(S_0)>\max_{t\in [0,\omega]} \bigg (d(t)e^{\int_{t-\tau_i}^t d(r)\, dr}\bigg)\q {\rm for}\q  i=1,\dots,n\end{equation}
 or \begin{equation}\label{h4.1} p_i(S_0)\ge\omega^{-1} \Big (e^{\int_0^\omega d(r)\, dr}-1\Big) \max_{t\in [0,\omega]}e^{\int_{t}^{t+\tau_i} d(r)\, dr}\q {\rm for}\q  i=1,\dots,n.\end{equation}\end{cor}

\begin{cor}\label{cor3.4}  Consider the chemostat model
 \begin{equation}\label{Chem0_one}
 \begin{split}
 S'(t)&=d(t)(S_0-S(t))-\sum_{i=1}^n p(S(t))x(t)\\
x'(t)&=-d(t)x(t)+e^{-\int_{t-\tau}^td(s)\, ds}p(S(t-\tau))x(t-\tau),
\end{split}
\end{equation}
with $d(t)\in C_\omega ^+(\R)$ and positive, $S_0,\tau>0$ and $p_1(x):=p(x)$ satisfies (H2). If$$p(S_0)\min_{t\in[0,\omega]}\int_{t}^{t+\omega}e^{\int _{t+\tau}^{s} d(r)\, dr}\, ds>e^{\int_0^\omega d(r)\, dr}-1,$$
 there \eqref{Chem0_one} possesses a  positive $\omega$-periodic solution $(S^*(t),x^*(t))$. 
% In particular, this is the case if 
% \begin{equation*}\label{h3'}\dps p(S_0)>\max_{t\in [0,\omega]} \bigg (d(t)e^{\int_{t-\tau}^t d(r)\, dr}\bigg).\end{equation*}
 \end{cor}

 \section{Competitive exclusion, coexistence, attractivity}
\setcounter{equation}{0}

 In Corollary \ref{cor3.1}, criteria for the existence of $n$ different  $\omega$-periodic  solutions of \eqref{Chem}, of the forms
 $X_1^*(t)=(S_1^*(t),x_1^*(t),0,\dots, 0),\dots, X_n^*(t)=(S_n^*(t),0,\dots, 0,x_n^*(t)) $ with $(S_i^*(t),x_i^*(t))>0\, (1\le i\le n)$, were provided. In this situation, it would be relevant to have sufficient conditions for either   the attractivity of one of such solutions, say $(S_i^*(t),0,\dots, x_i^*(t),0,\dots, 0)$, leading to the  survival in a sustained periodic way of species $i$ and the extinction of all the others, or for the coexistence of all or part of the competitive species.

 The first  goal is to address the global attractivity of a periodic solution when
 competitive exclusion occurs and only the  fittest species, say species $i$,  survives. For simplicity of writing, we fix $i=1$.
 
% For \eqref{Chem}  with $S_0(t),d(t)$ positive $\omega$-periodic functions and under (H2), 
% %$S_0(t),d(t)$ positive $\omega$-periodic functions, and $p_i(x)$ satisfying   (H2).
%suppose e.g. that (H3)(a) is satisfied for $i=1$ and \eqref{exti} holds for $i=2, \dots, n$: 
% \begin{equation}\label{survival1}
% \begin{split}
% p_1( y^*(t))&>d(t+\tau_1)e^{\int_t^{t+\tau_1} d(r)\, dr},\q  t\in [0,\omega],\\
%p_i\big( y^*(t)\big)&<d(t+\tau_2)e^{\int^{t+\tau_2}_t d(r)\, dr},\q  t\in [0,\omega], i=2,\dots,n.
%\end{split}
%\end{equation}
%From
% Theorem \ref{thm3.1} the competing $i$-species with $i=2, \dots,n,$ will not survive, since   $x_i(t)\to 0$ as $t\to\infty$  for all  solutions $(S(t),x_1(t),x_2(t),\dots,x_n(t))$ with initial conditions in  $C_0^+$.
%

\begin{thm}\label{thm4.1} For system \eqref{Chem}, assume (H1)-(H2). If
 \begin{equation}\label{survival1}
 \begin{split}
 p_1( y^*(t))&>d(t+\tau_1)e^{\int_t^{t+\tau_1} d(r)\, dr},\q  t\in [0,\omega],\\
p_j\big( y^*(t)\big)&<d(t+\tau_j)e^{\int^{t+\tau_j}_t d(r)\, dr},\q  t\in [0,\omega], j=2,\dots,n,
\end{split}
\end{equation}
and $p_1(x)$ is strictly increasing, 
then \eqref{Chem} admits a  nontrivial $\omega$-periodic solution of the form $(S^*(t),x_1^*(t),0,\dots,0)$, with $0<S^*(t)<y^*(t),x_1^*(t)>0$ for $t\in[0,\omega]$, which is a global attractor of all solutions with initial conditions in $C_0^+$. \end{thm}

 \begin{proof} From Theorem \ref{thm3.4},  \eqref{Chem} admits a nontrivial $\omega$-periodic solution of the form $$X_1^*(t)=(S^*(t),x_1^*(t),0,\dots,0),$$ with $0<S^*(t)<y^*(t),x_1^*(t)>0$ for $t\in[0,\omega]$. 
 
 Fix any  solution $(S(t),x_1(t),x_2(t),\dots,x_n(t))$   of  \eqref{Chem} with initial conditions in $C_0^+$.  Since \eqref{survival1} holds, Theorem 3.3 in\cite{Faria23} implies that  $\liminf_{t\to \infty} x_1(t)>0$ and $x_i(t)\to 0$ as $t\to\infty$ for  $i=2, \dots,n$.
% The competing $i$-species with $i=2, \dots,n,$ will not survive, since 
% Theorem \ref{thm3.1} leads to  $x_i(t)\to 0$ as $t\to\infty$ for  $i=2, \dots,n$.
%On the other hand, Lemma \ref{thm3.2} yields  $\liminf_{t\to \infty} x_1(t)>0$. 

 Denote $y^*(t)-e^{\int_t^{t+\tau_1} d(r)\, dr}x^*_1(t+\tau_1)=:y_1^*(t)$. 
 Effecting the change of variables $z_1(t)=\frac{x_1(t)}{x_1^*(t)}-1$,
 %, z_i(t)=x_i(t)\, (2\le i\le n)$, 
the first equation in \eqref{Chem_y} is transformed into
 \begin{equation}\label{z1}
 \begin{split}
 z_1'(t)
 =e^{-\int_{t-\tau_1}^t d(r)\, dr}&
 \bigg [ p_1\bigg (y_1^*(t-\tau_1)-e^{\int_{t-\tau_1}^t d(r)\, dr}x_1^*(t)z_1(t)- h_1(t)\bigg)(1+z_1(t-\tau_1))\\
 &- p_1 \big (y_1^*(t-\tau_1)\big )(1+z_1(t))\bigg ] \frac{x_1^*(t-\tau_1)}{x^*(t)},
 \end{split}
\end{equation}
where $h_1(t):=\sum_{i=2}^ne^{\int_{t-\tau_1}^{t+\tau_i-\tau_1} d(s)\, ds}x_i(t+\tau_i-\tau_1)$. Define
$$  -v=\liminf_{t\to\infty} z_1(t),\q u=\limsup_{t\to\infty} z_1(t).$$
From the permanence of $x_1(t)$,  $-1<-v\le u<\infty$. It suffices to show that $U:=\max \{u,v\}=0$. 

For the sake of contradiction, suppose that $U>0$. We only treat the case $U=u$  (the situation $U=v$ is similar). From the fluctuation lemma, there exists $t_k\to\infty$ with $z_1(t_k)\to U$ and $z_1'(t_k)\to 0$, as $k\to\infty$. Clearly $h_1(t_k)\to 0$ and, for some subsequence still denoted by $(t_k)$, we have
$$e^{\int_{t_k-\tau_1}^{t_k} d(r)\, dr}\to \hat e,\  y_1^*(t_k-\tau_1)\to \hat  y_\tau,\ x_1^*(t_k)\to \hat x^*, \ x_1^*(t_k-\tau_1)\to \hat x_\tau^*, \ z_1(t_k-\tau_1)\to \hat z_\tau.
$$
Using \eqref{z1} and taking limits lead to
\begin{equation*}
 \begin{split}
 0
 &=\hat e^{-1}
 \Big [ p_1\big (\hat y_\tau-\hat e \hat x^* U\big)(1+\hat z_\tau)
 - p_1 \big (\hat y_\tau)\big )(1+U)\Big ] \frac{\hat x_\tau^*}{\hat x^*}\\
 &\le  \hat e^{-1}(1+U)\Big [ p_1\big (\hat y_\tau-\hat e \hat x^* U\big)
 - p_1 \big (\hat y_\tau)\big )\Big ]\frac{\hat x_\tau^*}{\hat x^*}<0,
 \end{split}
\end{equation*}
which is a contradiction. Thus $U=0$ and the proof is complete.
 \end{proof}

 The question of coexistence of all the competing  species growing in the chemostat vessel  is far more challenging.  
A  first criterion asserting that there exists a positive $\omega$-periodic solution $x^*(t)$ is provided below. For simplicity and without loss of generality, since the order of the species may be exchanged, we state a result by setting  in the hypotheses an a priori ordering  for the species: $x_1,x_2,\dots,x_n$.

 \begin{thm}\label{thm4.2} Consider system \eqref{Chem} and assume (H1)-(H2). In addition, suppose that
 \begin{equation}\label{positive1a}
 \begin{split}
 %p_1( y^*(t))&>d(t+\tau_1)e^{\int_t^{t+\tau_1} d(r)\, dr},\q  t\in [0,\omega],\\
p_i\big( y_{i-1}^*(t)\big)&>d(t+\tau_i)e^{\int^{t+\tau_i}_t d(r)\, dr},\q  t\in [0,\omega], i=1,2,\dots,n,
\end{split}
\end{equation}
where:\\
(i) $y_0^*(t):=y^*(t)$ and $ x^{**}_{1}(t):=x^{*}_1(t)$;\\
(ii) $y_{i}^*$ are defined as\begin{equation*} 
y_{i}^*(t):=y_{i-1}^*(t)-e^{\int_t^{t+\tau_{i}} d(r)\, dr}x^{**}_{i}(t+\tau_i),\q i= 1,\dots, n-1;
\end{equation*}
(iii)  $x_i^{**}(t)$ is a positive $\omega$-periodic solution of
 \begin{equation}\label{Chem_xi}
x_i'(t)=-d(t)x_i(t)+e^{-\int_{t-\tau_i}^td(s)\, ds}p_i\big(y_{i-1}^*(t-\tau_i)-e^{\int_{t-\tau_i}^td(s)\, ds}x_i(t)\big)x_i(t-\tau_i),\ \ i=1,\dots, n.
\end{equation} 
Then \eqref{Chem} admits a positive $\omega$-periodic solution. \end{thm}
 
 \begin{proof} 
 
 For  \eqref{Chem}, or alternatively \eqref{Chem_y},  \eqref{positive1a} asserts that (H3.a) is satisfied for $i=1$:
$$ p_1( y^*(t))>d(t+\tau_1)e^{\int_t^{t+\tau_1} d(r)\, dr}, \q  t\in [0,\omega].$$
%$S_0(t),d(t)$ positive $\omega$-periodic functions, and $p_1(x),p_2(x)$ satisfying   (H2). In addition, assume that (H3)(a) is satisfied for $i=1$ and \eqref{exti} holds for $i=2$:
% \begin{equation}\label{n=2}
% \begin{split}
% p_1( y^*(t))&>d(t+\tau_1)e^{\int_t^{t+\tau_1} d(r)\, dr},\q  t\in [0,\omega],\\
%p_2\big( y^*(t)\big)&<d(t+\tau_2)e^{\int^{t+\tau_2}_t d(r)\, dr},\q  t\in [0,\omega].
%\end{split}
%\end{equation}
%In this case, the competing species with $i=2$ will not survive, since from
% Theorem \ref{thm3.3}  $x_2(t)\to 0$ as $t\to\infty$ for all  solutions $(S(t),x_1(t),x_2(t))$ with initial conditions in  $C_0^+$.
Consider  a periodic solution $X_1^*(t)=(S_1^*(t),x_1^*(t),0,\dots,0)$ with $x_1^*(t)>0$, whose existence is established in Theorem \ref{thm3.4} and Corollary \ref{cor3.1}. Denote $y^*(t)-e^{\int_t^{t+\tau_1} d(r)\, dr}x^*_1(t+\tau_1)=:y_1^*(t)$. From \eqref{positive1a}, we have
  \begin{equation}\label{n=2_coex}
 \begin{split}
p_2\big( y_1^*(t)\big)&>d(t+\tau_2)e^{\int^{t+\tau_2}_t d(r)\, dr},\q  t\in [0,\omega].
\end{split}
\end{equation}
  Theorem \ref{thm3.4}  now implies that there exists a positive $\omega$-periodic solution $x_2^{**}(t)$ of the equation
    \begin{equation}\label{x2}x_2'(t)=-d(t)x_2(t)+e^{-\int_{t-\tau_2}^t d(r)\, dr}p_2\big (y_1^*(t-\tau_2)-e^{\int_{t-\tau_2}^t d(r)\, dr}x_2(t)\big)x_2(t-\tau_2).\end{equation}
 Therefore,  $X_2^{**}(t)=(S_2^{**}(t),x_1^*(t),x_2^{**}(t),0,\dots,0)$,  where $S_2^{**}(t)=y_1^*(t)-e^{\int^{t+\tau_2}_t d(r)\, dr}x_2^{**}(t+\tau_2)$, is a  positive  $\omega$-periodic solution of \eqref{Chem}.  
This argument may be pursued in an iterative way for any $n>2$, thus we find a strictly positive $\omega$-periodic solution of \eqref{Chem}.\end{proof}
 
% ---
% 
% 
% [TIRAR: Since (H3.a) is satisfied for $i=1$, take
% a periodic solution $(S_1^*(t),x_1^*(t),0,\dots,0)$ with $x_1^*(t)>0$. For  $y_1^*(t)=y^*(t)-e^{\int_t^{t+\tau_1} d(r)\, dr}x^*_1(t+\tau_1)$, under \eqref{positive1a} and proceeding as above we find a solution  $X_2^{**}(t)=(S_2^{**}(t),x_1^*(t),x_2^{**}(t),0,\cdots,0)$ with the first three components positive. For
% $y_2^*(t)=y_1^*(t)-e^{\int^{t+\tau_2}_t d(r)\, dr}x_2^{**}(t+\tau_2)$, from \eqref{positive1a} we have $p_3\big( y_2^*(t)\big)>d(t+\tau_3)e^{\int^{t+\tau_3}_t d(r)\, dr},\  t\in [0,\omega]$, so that we find a solution
% $X_3^{**}(t)=(S_3^{**}(t),x_1^*(t),x_2^{**}(t),x_3^{**}(t),0,\cdots,0)$ with $S_3^{**}(t), x_i^{**}(t)>0, i=1,2,3.$
%  Iterating the procedure, we find a strictly positive $\omega$-periodic solution of \eqref{Chem}.]
  
  \begin{rmk} {\rm With the above notations, this criterion is still valid with conditions \eqref{positive1a} replaced by
   \begin{equation}\label{positive1c}
\int_0^\omega p_i( y_{i-1}^*(t))\, dt \ge \Big (e^{\int_0^\omega d(r)\, dr}-1\Big) \max_{t\in [0,\omega]}e^{\int_{t}^{t+\tau_i} d(r)\, dr}, \q  i=2,\dots,n.
\end{equation}
 For related results, see  \cite[Theorem 3.1]{WZ} for an ODE chemostat model with $n$ species in competition and no delays, where it was  asserted  that a strictly positive periodic solution must exist under  conditions similar to \eqref{positive1c}. %[Drop? This seems to be repeated...]
 }
  \end{rmk} 
  
Rather   then \eqref{positive1a}  or \eqref{positive1c}, it would be relevant to find more natural and easier verifiable conditions for the co-existence of all or part of the microorganisms. Note that, even under (H1)-(H2) and the uniform persistence of \eqref{Chem} (in $C_0^+$) -- as mentioned, a criterion was given  in \cite{Faria23} --, it is not obvious that the uniform persistence of all components of solutions with initial conditions in $C_0^+$ guarantees the existence of an $\omega$-periodic solution $x^*(t)$  whose  components are all  positive.  We shall prove that this is in fact true by using a fixed point theorem in \cite{Zhao}. As a first step, a preliminary result establishes a ``box" where to look for such a positive $\omega$-periodic solution.

  \begin{thm}\label{thm4.3} For system \eqref{Chem} with $S_0,d, p_i,\tau_i\, (1\le i\le n)$  satisfying (H1)-(H3), consider $n$  nonnegative $\omega$-periodic solutions of the form $X_i^*(t):=(S_i^*(t),0,\dots, 0,x_i^*(t), 0,\dots, 0)$ with $0<S_i^*(t)<y^*(t)$ and $x_i^*(t)>0$ ($1\le i\le n$) as in Corollary \ref{cor3.1}. Then the ordered set 
  $$\prod_{i=1}^n [0,x_i^*]=\{\var=(\var_1,\dots,\var_n)\in \mathfrak{C}:0\le \var_i\le x_i^*, i=1,\dots,n\}$$
  is positively invariant for \eqref{Chem_y}. 
%  Consequently, for any solution $X(t)=(S(t),x(t))$ of \eqref{Chem} with initial conditions $X_0\in \big [0, \ol{y^*}\big ]\times\prod_{i=1}^n [0,x_i^*]$ there exists a periodic solution $p(t)\ge 0$ of \eqref{Chem} such that $\lim_{t\to\infty} |X(t)-p(t)|=0.$
  \end{thm}
  
  \begin{proof}   Solutions $x(t)\ge 0$ of \eqref{Chem_y}  satisfy
   \begin{equation}\label{4.5}
 \begin{split}
 %y'(t)&=d(t)(S_0(t)-y(t))\\
x_i'(t)&\le -d(t)x_i(t)\\
&+e^{-\int_{t-\tau_i}^td(s)\, ds}p_i\Big(y^*(t-\tau_i)- e^{\int_{t-\tau_i}^{t} d(s)\, ds}x_i(t)\Big)x_i(t-\tau_i),\ i=1,\dots,n.\\
\end{split}
\end{equation}
Moreover the system of decoupled DDEs
 \begin{equation}\label{4.6}
 \begin{split}
 %y'(t)&=d(t)(S_0(t)-y(t))\\
u_i'(t)&= -d(t)u_i(t)\\
&+e^{-\int_{t-\tau_i}^td(s)\, ds}p_i\Big(y^*(t-\tau_i)- e^{\int_{t-\tau_i}^{t} d(s)\, ds}u_i(t)\Big)u_i(t-\tau_i),\ i=1,\dots,n,\\
\end{split}
\end{equation}
    satisfies the quasi-monotone condition (Q) in \cite[p.~78]{Smith}.  Clearly, $(x_1^*(t),\dots,x_n^*(t))$ is a solution of \eqref{4.6}. Although \eqref{Chem_y} is not in general a DDE, but a {\it mixed-type} FDE in the abstract space $\mathfrak{C}$, for solutions with $x_i(t)\le x_i^*(t)$ for all $t\in [-\tau,\tau^0], i=1,\dots,n,$ we can apply the proof of  \cite[Theorem 5.1.1]{Smith} and deduce
 that any solution $x(t)$ with initial condition $\var\in X:=\prod_{i=1}^n [0,x_i^*]$ satisfies $x_t\in X$ for all $t\ge 0$.  The invariance of the  set $\prod_{i=1}^n [0,x_i^*]$ is proven.

%
%The last assertion in the statement follows from the fact that the Poincar\'e map $\pi:X\to X, \var\mapsto x_\omega(\cdot,\var)$, satisfies $0\le P\var\le \var$ for all $\var \in X$, therefore $P^n\var$ converges to a fixed point $\var^*\in X$ of $\pi$. Therefore,  $x(\cdot; \var^*)$ is a periodic solution of \eqref{Chem_y} and  $\lim_{t\to\infty} |x(t;\var)-x(t; \var^*)|=0$. Using \eqref{y} and \eqref{y_ode}, we obtain the desirable result,
%with $p(t)=(S^*(t),x(t; \var^*))$ and
% $S^*(t)=y^*(t)-\sum_{i=1}^n e^{\int_t^{t+\tau_i} d(s)\, ds}x_i(t+\tau_i;\var^*)$.
   \end{proof}

   Observe that the Krasnoselskii theorem allows finding a fixed point $x^*(t)$ of the operator $\Phi$ in \eqref{Phi} in the set $K\cap \prod_{i=1}^n [0,x_i^*]$, however it does not provide the positivity of all components of $x^*(t)$.    
 For a result in the framework of ODE chemostat models with two species in competition, see \cite[Corollary 7.5.2]{SmithW}, however the situation is quite different for delayed chemostat models.
 In spite of this, 
 a criterion for  the existence of  a  strictly positive solution of \eqref{Chem_y} in $C_\omega\cap \prod_{i=1}^n [0,x_i^*]$ is now given.

     \begin{thm}\label{thm4.4} Consider system \eqref{Chem} with $S_0,d, p_i,\tau_i\, (1\le i\le n)$ satisfying (H1)-(H3). Assume also that  \eqref{Chem}  is uniformly persistent (in $C_0^+$). Then there exists a strictly positive $\omega$-periodic solution of  \eqref{Chem}.
      \end{thm}
      
      \begin{proof} Consider $\var\in X:=\prod_{i=1}^n [0,x_i^*]$ as a subset of the Banach space $\mathfrak{C}$ with the supremum norm $\|\cdot\|_\infty$. With the induced metric, $X$ is complete. Consider also $X_0:=\prod_{i=1}^n (0,x_i^*]$ as a subset of $X$. Note that $X_0$ is convex,  relatively open in $X$ and its border $\p X_0$  in $X$ is the set of functions $ \var=(\var_1,\dots,\var_n)\in [0,x_i^*]$ such that $\var_i(t)=0$ for some $i\in\{1,\dots,n\}$ and some $t\in [-\tau,\tau^0]$.
      
      From Theorem \ref{thm4.3}, $X$ is positively  invariant for  \eqref{Chem_y}. Consider the Poincar\'e map
      $$P:X\to X,\q \var\mapsto x_\omega(\cdot,\var),$$
      where $x(t,\var)$ is the solution of  \eqref{Chem_y} with initial condition $x_0=\var$. Since solutions with initial condition in $C_0^+$ are  positive for $t\ge 0$, then $X_0$ is positively  invariant  as well, and in particular $P(X_0)\subset X_0$. From the assumptions, system \eqref{Chem} is uniformly persistent in $C_0^+$, therefore there exists $\eta>0$ such that
      $$\liminf_{n\to\infty}d(P^n(\var), \p X_0)\ge \eta\q {\rm for\ all}\q \var \in X_0.$$
      As $P(X)\subset X$ and $X$ is uniformly bounded from above, then
   $\|x_t(\var)\|\le C$ for all $\var\in X$ and $t\ge 0$, for $C=\max_{1\le i\le n}\|x_i^*\|$. 
      
      Next, we prove that $P$ is a compact operator. For $x(t)=x(t,\var)$ and $P=(P_1,\dots,P_n)$, we have
   \begin{equation}\label{Poinc1}
   \begin{split}(P_ix)(t)&=x_{i\omega}(t)=x_i(t+\omega)\\
   &=\var_i(0)e^{-\int_0^{t+\omega}d(r)\, dr}\\&+\int_0^{t+\omega}e^{-\int^{t+\omega}_{s-\tau_i}d(r)\, dr}p_i\Big( y^*(s-\tau_i)-\sum_{j=1}^n e^{\int_{s-\tau_i}^{s+\tau_j-\tau_i} d(r)\, dr}x_j(s+\tau_j-\tau_i)\Big) x_i(s-\tau_i)ds
   \end{split}
     \end{equation}
     for $t\in I:=[-\tau,\tau^0]$. Writing $D(t)=e^{\int_0^{t}d(r)\, dr}$, for $t_1,t_2\in I, t_1>t_2$ we derive
  \begin{equation}\label{Poinc2}
   \begin{split}|(P_ix)(t_1)-(P_ix)(t_2)|
   &\le Ce^{-D(\omega)}\big(e^{-D(t_1)}-e^{-D(t_2)}\big)\\&+Ce^{-D(\omega)}p_i\big(\ol{S_0}\big)\int_0^{t_2+\omega}\big(e^{-\int^{t_2}_{s-\tau_i}d(r)\, dr}-e^{-\int^{t_1}_{s-\tau_i}d(r)\, dr}\big)\, ds\\
   &+Ce^{-D(\omega)}p_i\big(\ol{S_0}\big)\int_{t_2+\omega}^{t_1+\omega}e^{-\int^{t_1}_{s-\tau_i}d(r)\, dr}\, ds.
   \end{split}
     \end{equation}
 From the uniform continuity and boundedness of $e^{-D(t)}$ on the compact interval $I$, it follows that $P(X)$ is equicontinuous.  From Ascoli-Arzel\`a Theorem, $P$ is a compact operator.
      
      In this way, all the assumptions of \cite[Theorem 2.1]{Zhao} are satisfied,  which allows us to conclude that  $P$ has a fixed point $\var^\#$ in $X_0$. Thus $$\Big( y^*(t)-\sum_{i=1}^n e^{\int_t^{t+\tau_i} d(s)\, ds}x_i^\#(t+\tau_i),x_1^\#(t),\dots,x_n^\#(t)\Big),$$ where $x(t,\var^\#):=x^\#(t)=(x_1^\#(t),\dots,x_n^\#(t))$, is a positive $\omega$-periodic solution of  \eqref{Chem}. \end{proof}
      
       \begin{rmk}\label{rmk4.2} {\rm Consider system \eqref{Chem}  under hypotheses (H1)-(H2). As shown in \cite{Faria23}, if \eqref{pers}  holds, where $m_0>0$ is a uniform lower bound for the coordinate $S(t)$ of solutions  of \eqref{Chem} with initial conditions in $C_0^+$, then the system is permanent. Moreover, \eqref{pers} implies (H3). From \cite{Faria23}, a rough estimate for $m_0$ is also provided. However, in general that estimate is not compatible with \eqref{pers}, and therefore is not useful for applying the above Theorem \ref{thm4.4}.  To clarify this assertion, in the next example the case of $d(t)$ constant is considered.}
      % Then there exists a strictly positive $\omega$-periodic solution of  \eqref{Chem}.
      \end{rmk}

%---
%      
%      [Write corollaries for $d(t)\equiv d>0$.?]
%
%   
%
% ---
 
% \

 \begin{exmp}\label{examp4.0}  {\rm  Consider  \eqref{Chem} with $d(t)\equiv d>0$, $S_0(t)$  not constant and $S_0(t), p_i(t)$
 satisfying (H1)-(H2). With the above notations, from \cite[Lemma 2.3]{Faria23}, we obtain the lower estimate $S^-:=\liminf_{t\to\infty} S(t)\ge m_0$
%\begin{equation}\label{pers_S}\liminf_{t\to\infty} S(t)\ge m_0,\end{equation}
where 
\begin{equation}\label{pers_S}m_0=\min \{s\in [0,\ol{S_0}]: h(s)=d \ul{S_0}\},\end{equation} for $h(s):=ds+(\ol{S_0}-s) \sum_{i=1}^np_i(s)e^{-d\tau_i}.$ Conditions \eqref{pers} and (H3.a) translate respectively as
$$p_i(m_0)\ge de^{d\tau_i}\q {\rm and}\q
\dps\min_{t\in [0,\omega]}  p_i( y^*(t))>de^{d\tau_i}$$
for $ i=1,\dots,n,$ and clearly the first inequality implies the second.
Assume next  that $p_i(m_0)\ge de^{d\tau_i}, i=1,\dots,n.$
 Then, for any $n\ge 1$, we have $$h(m_0)>dm_0+nd(\ol{S_0}-m_0)\ge d\ol{S_0}>d\ul{S_0},$$
therefore condition \eqref{pers_S} is not fulfilled.   This suggests the need for other criteria for persistence, leading to a sharper uniform lower bound $S^-$.}\end{exmp}

 \begin{exmp}\label{examp4.1}  {\rm  Consider  \eqref{Chem} with $d(t)\equiv d>0, S_0(t)=a+\sin t$ for some $a>1$, and $p_i(x)$ satisfying (H2).
 The $2\pi$-positive solution of \eqref{y_ode} is
 \begin{equation}\label{y_exam}
 y^*(t)=a+\frac{d}{d^2+1}\big (d\sin t-\cos t\big ).
 \end{equation}
Maxima and minima of $y^*(t)$ are attained at points such that $\sin t=-d\cos t$, from which we deduce that $\max_\R y^*(t)=a+\frac{d}{\sqrt{d^2+1}}, \min_\R y^*(t)=a-\frac{d}{\sqrt{d^2+1}}$. From Theorem \ref{thm2.1}, if
 \begin{equation}\label{exti_exp}p_i\Big(a+\frac{d}{\sqrt{d^2+1}}\Big)\le de^{d\tau_i},\q i=1,\dots,n,\end{equation}
 then the washout $2\pi$-periodic solution $(y^*(t),0,\dots,0)$ is GAS.
 
 We next derive some conditions for uniform persistence in $C_0^+$.
Suppose that there exists $(p_i)_+'(0)=b_i>0$ and that for any $\vare\in (0,1)$ there is $\de \in (0,\frac{a-1}{2})$ such that $p_i(s)\le b_i s$ for $s\in [0,a+\frac{d}{\sqrt{d^2+1}}]$ and
 \begin{equation}\label{p_iEx1}p_i(s)\ge b_i(1-\vare) s \q {\rm for}\q s\in [0,\de], i=1,\dots,n. \end{equation}Let $m_0>0$ be a uniform lower bound for the first coordinate $S(t)$ of positive solutions. From \cite[Lemma 2.3]{Faria23}, we can estimate $m_0$ as the minimum of the positive $s\in [a-\frac{d}{\sqrt{d^2+1}},a+\frac{d}{\sqrt{d^2+1}}]$ such that
$h(s):=ds+(a+\frac{d}{\sqrt{d^2+1}}-s)\sum_{i=1}^n p_i(s)e^{-d\tau_i}=d(a-\frac{d}{\sqrt{d^2+1}}).$ For $h^+(s):=ds\Big [ 1+(a+\frac{d}{\sqrt{d^2+1}}-s)B\Big ]$  where $B=d^{-1}\sum_i b_ie^{-d\tau_i}$, we have $h^+(s)\ge h(s)$ on $[0,a+\frac{d}{\sqrt{d^2+1}}]$.
In particular,   we obtain $m_0\ge m_0^+>0$, where $m_0^+>0$ is such that $h^+(m_0^+)=d(a-\frac{d}{\sqrt{d^2+1}})$, thus
 \begin{equation}\label{m_0}
 m_0\ge \frac{a-\frac{d}{d^2+1}}{1+B (a+\frac{d}{d^2+1})}.\end{equation}%=\frac{a(d^2+1)-d}{(d^2+1)+B [a(d^2+1)+d]}.$$
If
$(b_i-\vare)\min \{\de,m_0\}\ge de^{d\tau_i},  i=1,\dots, n,$
 from \eqref{p_iEx1} it follows that $p_i(m_0)>d e^{d\tau_i},1\le i\le n,$ so 
 \eqref{Chem} is persistent \cite{Faria23}.  
 On the other hand, if
 $$p_i\Big(a-\frac{d}{\sqrt{d^2+1}}\Big)> de^{d\tau_i},\q i=1,\dots,n,$$
 respectively
 $$\int_0^{2\pi}p_i\Big (a+\frac{d}{d^2+1}\big (d\sin t-\cos t\big )\Big )\, dt>e^{d\tau_i}(e^{2\pi d}-1),\q i=1,\dots,n,$$
holds,  then (H3.a), respectively (H3.b), is satisfied, and Theorem \ref{thm3.3} yields the existence of at least $n$ $2\pi$-periodic solutions  $X_1^*(t)=(S_1^*(t),x_1^*(t),0,\dots,0),\dots, X_n^*(t)=(S_n^*(t),0,\dots,0,x_n^*(t))$, with $S_i^*(t),x_i^*(t)>0, i=1,\dots,n, t\in [0,2\pi]$. 
Consider now  the particular case of the example above with $d=1, a=2$, so that \eqref{Chem} is given by
  \begin{equation}\label{ChemEx2}
 \begin{split}
 S'(t)&=2+\sin t-S(t)-\sum_{i=1}^np_i(S(t))x_i(t)\\
x_i'(t)&=-x_i(t)+e^{-\tau_i}p_i(S(t-\tau_i))x_i(t-\tau_i),\q i=1,\dots,n,\\
\end{split}
\end{equation}  
and $p_i(x)=\frac{b_ix}{1+x}$ for some $b_i>0$. 
The $2\pi$-positive solution of \eqref{y_ode} is
 $y^*(t)=2+\frac{1}{2}\big (\sin t-\cos t\big ).$ Note that $p_i(2-\frac{1}{\sqrt 2})=b_i\frac{11-\sqrt 2}{17}$. 
 If 
% $$
%  b_1  \frac{2a-\sqrt 2}{2+2a-\sqrt 2}\ge e^{\tau_1}\q {\rm and}\q  b_j  \frac{2a+\sqrt 2}{2+2a+\sqrt 2}\le e^{\tau_j}\ {\rm for}\ j=2,\dots, n,$$
 $\tau_i<\log b_i+\log (\frac{11-\sqrt 2}{17}) \approx \log b_i-0.5729$ for $i=1,\dots, n,$
 then (H3.a) holds and  \eqref{ChemEx2} has at least $n$  $2\pi$-periodic solutions of the form
  $X_i^*(t):=(S_i^*(t),0,\dots, 0,x_i^*(t), 0,\dots, 0)$ with $0<S_i^*(t)<y^*(t)$ and $x_i^*(t)>0$ ($1\le i\le n$). In particular, if $n=1$ and $\tau_1<\log b_1+\log (\frac{11-\sqrt 2}{17})$, there is a positive $2\pi$-periodic solution $(S(t),x_1(t))$.
  
%  ----
%  [Below: Correct or Delete!]
%  
%  Next, for \eqref{ChemEx2} with $n\ge 2$, rather than the estimate in \eqref{m_0}, we use directly the definition of $h(s)$ above to conclude from \cite[Lemma 2.3]{Faria23}  that the uniform lower bound $m_0\in [2-\frac{1}{\sqrt 2}, 2+\frac{1}{\sqrt 2}]$ for the first coordinate $S(t)$  satisfies $m_0\ge s$ where $s$ is given by the equality 
% $$(1-B)m_0+B(2+\frac{1}{\sqrt{2}})\frac{m_0}{1+m_0}=2-\frac{1}{\sqrt{2}}
% $$
% for $B=\sum_{j=1}^n b_je^{-\tau_j}$.
% 
% For a concrete illustration, consider the case of \eqref{ChemEx2} with $n=2$ and $B=10$. Then $m_0=2+\frac{11}{9\sqrt 2}$ and $\frac{m_0}{1+m_0}\approx 0.7412$. Conditions  \eqref{pers} translate as
%
% 
% $$(B-1)m_0^2+\Big [1-\frac{1}{\sqrt 2}-B(2+\frac{1}{\sqrt 2})\Big ]m_0+2-\frac{1}{\sqrt 2}=0.$$
%  
% 
%  . However, for $n\ge 2$, it is not possible to have $p_i(\frac{11-\sqrt 2}{10B})>e^{\tau_i}$ for $i=1,\dots, n$, since this would imply $b_ie^{-\tau_i} \frac{11-\sqrt 2}{11-\sqrt 2+10B}>1$ for all $i$ and therefore
%  $\sum_{i=1}^nb_ie^{-\tau_i} \frac{11-\sqrt 2}{11-\sqrt 2+10B}=B\frac{11-\sqrt 2}{11-\sqrt 2+10B}>n$, which is not possible.
}
    \end{exmp} 
%\
%
%\
%$p_i\Big(a-\frac{d}{\sqrt{d^2+1}}\Big)> de^{d\tau_i},\q i=1,\dots,n,$
%    
%( 
%    \eqref{Chem} with $d(t)\equiv d>0, S_0(t)=a+\sin t$ for some $a>1$, and $p_i(x)$
% 
% $$b_i(1-\vare)\frac{a(d^2+1)-d}{(d^2+1)+B [a(d^2+1)+d]}>de^{d\tau_i},\q i=1,\dots,n,$$
% 
%% On the other hand, from the definition of $m_0$ and \eqref{S(t)}, we have
%% $$a-\frac{d}{\sqrt{d^2+1}}-m_0\le \sum_{i=1}^n e^{d\tau_i}\liminf_{t\to \infty}x_i(t) \le \sum_{i=1}^n e^{d\tau_i}\limsup_{t\to \infty} x_i(t)\le a+\frac{d}{\sqrt{d^2+1}}-m_0.$$
%% Reasing as in \cite[Lemma 2.3]{Faria23}, we deduce that
%% 
% 
% (continue/change computations... )
% 
%Try with $d=1, a=2$:
%
%
% $s+(1.807-s)\sum_{i=1}^n p_i(s)e^{-\tau_i}= .3929
% \iff (\sum_{i=1}^n p_i(s)e^{-\tau_i}-1)s=- .3929+1.807\sum_{i=1}^n p_i(s)e^{-\tau_i} ??$
%   
 % $$
%  b_1  \frac{2a-\sqrt 2}{2+2a-\sqrt 2}\ge e^{\tau_1}\q {\rm and}\q  b_j  \frac{2a+\sqrt 2}{2+2a+\sqrt 2}\le e^{\tau_j}\ {\rm for}\ j=2,\dots, n,$$

 \section{Conclusions and discussion}
 
 Lately,   periodic chemostat models with delays  have begun receiving  some attention, and the topic of the existence of nonnegative periodic solutions been investigated.  However, authors have mostly considered  the case of a single species \cite{ARS20,ARS20b}, the work \cite{JiaZ} being an exception.  Here,  a periodic delayed chemostat model  with $n$ species in competition given by \eqref{Chem} was studied. System \eqref{Chem} was derived by  implementing a nonautonomous periodic nutrient input and a periodic rate in the classic Ellermeyer autonomous model \cite{Eller}, generalised to $n$ species. For the first time, a systematic study of
a delayed periodic chemostat model with multiple growing  microorganisms in competition is performed.  Questions of extinction of all species, competitive exclusion, persistence and co-existence are addressed in the present paper, nevertheless its major  goal   is to give sufficient conditions  for the existence of nontrivial nonnegative and positive  periodic solutions for  \eqref{Chem}.  Among other results,  criteria providing the existence of  nontrivial periodic solutions are given. 
% %, see the main results Theorems \ref{thm3.3} and \ref{thm4.4}. 
 After a major criterion given in Theorem \ref{thm3.3} for the existence of such $n$ nonnegative solutions, we further explore either comparison results for quasi-monotone FDEs in \cite{Smith} or
 a fixed point theorem in \cite{Zhao}, to establish in Theorems \ref{thm4.2} and \ref{thm4.4} that a strictly positive $\omega$-periodic solution must exist. When applied to the particular case of a single species ($n=1$), this setting largely improves the result in \cite{ARS20}.

 The novelty of the paper relies both on the techniques and results. 
  A key original technique to carry out this investigation is a {\it conservation principle}, established in \cite{Faria23} along the lines in \cite{SmithW}, which allows replacing the $n+1$-dimensional DDE \eqref{Chem} by an $n$-dimensional FDE of mixed type. 
% On the other hand, for the first time a systematic study of
%a delayed periodic chemostat model with multiple growing  microorganisms in competition is performed.  Among other results,  criteria providing the existence of nontrivial periodic solutions are given.
 %, whereas the literature typically  only considers model with a single microorganism.
 
 This line of research and its methodology  also give rise to several open problems to be addressed in the next future. Some of them have already been mentioned throughout the paper. Namely,  from Remark \ref{rmk4.2} and Example \ref{examp4.0} it is apparent that it is relevant to improve the criterion in  \cite[Lemma 2.3]{Faria23}, by providing a sharper uniform lower bound  for the component $S(t)$ of all positive solutions,  consequently leading to better criteria for persistence in $C_0^+$. On the other hand, to study   the partial extinction of some of the competing species, while others uniformly persist, is an important topic and a conjecture has been presented in \cite{Faria23}.

  % (conf. Remark \ref{rmk4.2}).
  
% Firstly, for chemostat systems of the form \eqref{Chem} with $d(t),S_0(t)$  non-periodic but bounded above and below by positive constants:
% 
% (i)  To improve the criterion in \cite[Theorem 3.2]{Faria23} for the global attractivity of washout solutions $(y(t),0,\dots,0)$, with $y(t)$ a positive solution of \eqref{y_ode}, by using refinements of the arguments in \cite{{NahRost}} (conf. Theorem \ref{thm2.1}).
% 
% (ii) To provide a sharper uniform lower bound for $S^-=\liminf_{t\to\infty} S(t)$ in \cite[Lemma 2.3]{Faria23} and
% improve the criterion, leading to better criteria for persistence (conf. Remark \ref{rnk4.2}).
% 
%  . A list of problems and ideas to address are given below
%
%  Conclusions:   
%  - Thm 2.1: refinement of extinction 
%  
%  - Thm 3.2 and Corol: exist of n periodic nonnegative sol of the form...
%  
%  - thm 4.1: competitive exclusion and GA of a periodic sol
%  
%  - Thm 4.2: criterion for exist of a positive periodic sol (not very good)
%  
%  -thm 4.4: better criterion for exist of a positive periodic sol (under persistence + cond of thm 3.2)
%  
%  - examples..
% 
%  Open problems:
%  

  Note that the uniform persistence is a main assumption  in Theorem \ref{thm4.4}, to use  \cite{Zhao} and derive the existence of at least one positive periodic solution $(S^\#(t),x_1^\#(t),\dots,x_n^\#(t))$ of \eqref{Chem}, or equivalently $x^\#(t):=(x_1^\#(t),\dots,x_n^\#(t))$ of \eqref{Chem_y}.
Under the hypotheses of Theorem \ref{thm4.4}, Theorem 2.1 in Zhao's paper  \cite{Zhao} also asserts that \eqref{Chem_y} has a  nonempty, compact and invariant global attractor in $C_0^+$. Whether such  positive $\omega$-periodic solution $x^\#(t)$ is unique and globally attractive in $C_0^+$ is  an interesting  open problem. For the very particular case of a planar periodic ODE  \eqref{Chem0} {\it without delays} (and $S_0(t)\equiv S_0$, $n=2$), the asymptotic behaviour of positive solutions -- including conditions for the global attractivity of a positive equilibrium -- was described in \cite[Chap.~7]{SmithW}, by using Floquet theory for ODEs and monotone theory for competitive and cooperative planar systems. Not only  do these results not  extend directly to DDEs, but also system \eqref{Chem} is far more general. 
Also, contrarily to the approach often used for ODEs or even autonomous DDEs,   the construction of a suitable Lyapunov functional to derive the global attractivity of such a solution $x^\#(t)$ in $C_0^+$  is a priori not feasible, in virtue of   its nonautonomous character. Possible alternative directions include  exploring the work of Mazenc and  Malisoff \cite{MM} on  new constructions of Lyapunov-Krasovskii functionals for nonautonomous DDEs, or the technique of pullback attractors of Caraballo et al. \cite{CHK}. Another challenge is to perform the analysis of extinction, partial extinction, persistence and existence of nonnegative/positive periodic solutions for the generalisation of \eqref{Chem} obtained by considering  periodic delays $\tau_i(t)$, instead of constant delays $\tau_i$. With  time-varying delays $\tau_i(t)$, the main tool of the conservation principle is no longer valid, thus new methods to tackle the problem will need to be developed.

%   [ Possible important references:  
%   
%   $\star$ Use the recent "new cones" of Jianhong Wu
%   
%   $\star$ Use recent papers of Yuming Chen on Lyapunov functionals?]
%
%
%
%
%
%---
%
%From \cite{Faria23}
%
%A key ingredient for the results  is the {\it conservation principle} established here, which allows reducing the $n+1$-dimensional DDE \eqref{Chem} to an $n$-dimensional functional differential equation of mixed type. The uniform persistence of all the species in competition relies on a criterion for persistence of multi-dimensional DDEs in \cite{Faria21}, whereas the persistence of a unique species and extinction of  all the others is derived by using results in \cite{BB16} for scalar DDEs with mixed monotonicity. 
% A relevant conjecture claiming
%  the sustained  survival of a certain number $m \in \{2, \dots, n\}$ of the competitive species and extinction of the remaining $n-m$ is also presented.
%

\end{document}